\documentclass[a4paper,10pt,reqno]{amsart}
\usepackage[centering,includeheadfoot,margin=2.5 cm]{geometry}
\usepackage{graphics,enumitem,epsfig,textcomp}
\usepackage{amsfonts,amsmath,amssymb,euscript,color,mathrsfs}
\usepackage[utf8]{inputenc}	

\usepackage{algorithm}
\usepackage{algorithmic}

\usepackage{caption}
\usepackage{enumitem}

\usepackage[nocompress]{cite}

\makeatletter 
\@mparswitchfalse%
\makeatother
\normalmarginpar 

\usepackage{mathtools}

\usepackage[foot]{amsaddr}

\usepackage{todonotes}
\usepackage{fp}

\usepackage{hyperref}
\hypersetup{
	colorlinks,
	citecolor=green,
	filecolor=black,
	linkcolor=blue,
	urlcolor=blue
}
\usepackage[all]{hypcap}
\usepackage[capitalise,noabbrev,sort]{cleveref}

\Crefname{assumption}{Assumption}{Assumptions}
\crefname{assumption}{Assumption}{Assumptions}

\newcommand*\D{\,\mathrm{d}}
\newcommand*\tol{\mathrm{tol}}
\newcommand*\res{\mathrm{res}}

\newcommand*\R{\mathbb{R}}

\newcommand*\N{\mathbb{N}}
\newcommand*\bl{\left(}
\newcommand*\br{\right)}
\newcommand*\AV[1]{| #1 |}
\newcommand*\norm[1]{\| #1 \|}
\newcommand*\bignorm[1]{\big\| #1 \big\|}
\newcommand*\Bignorm[1]{\Big\| #1 \Big\|}
\newcommand*\biggnorm[1]{\bigg\| #1 \bigg\|}

\newcommand*\ip[2]{\langle #1 \,, #2 \rangle}
\newcommand*\nablax{\nabla_{\!x}}
\newcommand*\scale[1]{#1_{\mathrm{sc}}}

\newcommand*\lipschitz{L}
\newcommand*\Tspace{\REV{L_2^T}}
\newcommand*\discr{\delta}
\newcommand*\poro{\phi}
\newcommand*\poroGeneral{\varphi}
\newcommand*\poroGeneralDiscr{\poroGeneral_\discr}
\newcommand*\poroOp{\Theta}
\newcommand*\poroOpLipschitz{\lipschitz_\poroGeneral}
\newcommand*\poroTransform{\lambda}
\newcommand*\press{u}
\newcommand*\pressAlt{v}

\newcommand*\pressFlux{\eta}
\newcommand*\pressFluxAlt{\mu}

\newcommand*\pressDiscr{\press_\discr}
\newcommand*\pressAltDiscr{\pressAlt_\discr}
\newcommand*\pressFluxDiscr{\pressFlux_\discr}
\newcommand*\pressFluxAltDiscr{\pressFluxAlt_\discr}
\newcommand*\pressBilin{\Lambda}
\newcommand*\pressSolSpDiscr{\pressSolSp_\discr}
\newcommand*\pressLinOp{G}
\newcommand*\pressLinRhs{R}
\newcommand*\pressNonLinOp{G}
\newcommand*\pressNonLinOpDiff{D\pressNonLinOp}
\newcommand*\pressNonLinOpInv{F}
\newcommand*\pressNonLinOpInvDiff{D\pressNonLinOpInv}
\newcommand*\pressRanSp{V}
\newcommand*\pressRhs{l}
\newcommand*\pressSolSp{U}
\newcommand*\Qalpha{\tilde{\alpha}}
\newcommand*\Qbeta{\tilde{\beta}}

\newcommand*\paraOp{\mathcal{P}}
\newcommand*\paraOpDiscr{\paraOp_\discr}
\newcommand*\interOp{\mathcal{I}}
\newcommand*\projOp{\Pi}
\newcommand*\general{\psi}
\newcommand*\generalDiscr[1][]{\general_\discr^{#1}}
\newcommand*\generalOp{\Xi}
\newcommand*\generalOpLipschitz{\lipschitz_\general}
\newcommand*\discrErr[1]{\epsilon_\discr^{(#1)}}

\newcommand*\pressOp{\Phi}
\newcommand*\pressOpLipschitz{\lipschitz_\press}
\newcommand*\errorRedFac{\xi}

\newcommand*\bulkVisc{\eta}
\newcommand*\densityDiff{\Delta\rho}
\newcommand*\gravity{g}
\newcommand*\perm{\kappa}
\newcommand*\fluidVisc{\mu}

\newcommand*\shapeFun{\mathcal{S}}
\newcommand*\polySp{\mathbb{P}}
\newcommand*\polySpMult{\mathbb{Q}}
\newcommand*\RT{\mathrm{RT}}

\newcommand*\thefont{\expandafter\string\the\font}

\renewcommand*\epsilon{\varepsilon}

\DeclareMathOperator*\argmin{arg\,min}

\DeclareMathOperator\Div{div}

\newtheorem{theorem}{Theorem}
\newtheorem{proposition}[theorem]{Proposition}
\newtheorem{lemma}[theorem]{Lemma}
\newtheorem{corollary}[theorem]{Corollary}

\newtheorem{assumption}{Assumption}
\theoremstyle{definition}
\newtheorem{definition}[theorem]{Definition}
\newtheorem{remark}[theorem]{Remark}

\newcommand*\convfac{0.8135937003116855}

\newcommand*\REV[1]{#1}

\numberwithin{theorem}{section}
\numberwithin{equation}{section}
		
\title[An Adaptive Space-Time Method for Nonlinear Poroviscoelastic Flows]{An Adaptive Space-Time Method for Nonlinear Poroviscoelastic Flows with Discontinuous Porosities}
\author{Markus Bachmayr$^\dagger$}
\email{bachmayr@igpm.rwth-aachen.de}
\author{Simon Boisser\'{e}e$^\dagger$}
\email{boisseree@igpm.rwth-aachen.de}
\address{$^\dagger$ Institut f\"{u}r Geometrie und Praktische Mathematik, RWTH Aachen University, Templergraben 55, 52056 Aachen, Germany}
\date{\today}

\thanks{M.B.\ acknowledges funding by Deutsche Forschungsgemeinschaft (DFG, German Research Foundation) -- project numbers 233630050, 442047500 -- TRR 146, SFB 1481.  S.B.\ has been funded in part by the M3ODEL consortium at Johannes Gutenberg University Mainz and by Deutsche Forschungsgemeinschaft -- project number 442047500 -- SFB 1481.}

\begin{document}

\maketitle

\begin{abstract}
    This paper is concerned with a space-time adaptive numerical method for instationary porous media flows with nonlinear interaction between porosity and pressure, with focus on problems with discontinuous initial porosities. A convergent method that yields computable error bounds is constructed by a combination of Picard iteration and a least-squares formulation. The adaptive scheme permits spatially variable time steps, which in numerical tests are shown to lead to efficient approximations of solutions with localized porosity waves. The method is also observed to exhibit optimal convergence with respect to the total number of spatio-temporal degrees of freedom.
\end{abstract}

\section{Introduction}
In porous media flows, important transient effects can arise from nonlinear interactions of porosity and pressure, which in certain cases can lead to the formation of \emph{porosity waves}. 
These can take the form of solitary waves formed by traveling higher-porosity regions \cite{Yarushina2015b} or of chimney-like channels \cite{Raess2019}. 
Such effects are important, for instance, in the modeling of rising magma \cite{McKenzie1984,Barcilon1986}, where porosity waves arise due to high temperatures. Such waves or channels can also form in soft sedimentary rocks, in salt formations or under the influence of chemical reactions; see for example \cite{Yarushina2015a,Raess2018,Raess2019}.
Quantifying uncertainties caused by the formation of preferential flow pathways can thus be important for safety analyses in geoengineering applications \cite{Yarushina2022}.

\subsection{Poroviscoelastic model}
We consider the instationary po\-ro\-vis\-co\-elas\-tic model analyzed in~\cite{Bachmayr2023} that can be regarded as a generalization of the models introduced in \cite{Connolly1998,Vasilyev1998} for the interaction of porosity $\phi$ and effective pressure $u$.
Throughout, we assume a spatial domain $\Omega\subseteq \R^d$ with $d \in \N$ to be given. For $T>0$, we write $\Omega_T=(0,T) \times \Omega$.
The model for a poroviscoelastic flow on which we focus in this work reads
\begin{subequations}\label{eq:model}
	\begin{align}
		\label{eq:modelphieq}	\partial_t \poro&=-(1-\poro) \bl \frac{b(\poro)}{\sigma(\press)}\press +Q\partial_t \press \br ,\\
		\label{eq:modelueq} \partial_t \press&=\frac{1}{Q} \bl \nabla \cdot a(\poro)(\nabla \press + (1-\poro) f)-\frac{b(\poro)}{\sigma(\press)}\press \br ,
	\end{align}
\end{subequations}
with functions $a$, $b$ and $\sigma$ that are to be specified, and where $Q>0$ and $f\in \R^d$ are assumed to be given constants. For details on the derivation of \eqref{eq:model}, we refer to \cite[Appendix~A]{Bachmayr2023}.
Physically meaningful solutions of this problem need to satisfy $\poro \in (0,1)$ on $\Omega_T$.
The problem is supplemented with initial data
\begin{align}\label{eq:ic}
	\poro(0,x)=\poro_0(x), \quad \press(0,x)=\press_0(x),\quad x\in\Omega,
\end{align}
for given functions $\poro_0\colon \Omega\to (0,1)$ and $\press_0\colon \Omega\to \R$, as well as homogeneous Dirichlet boundary conditions for $\press$ on $(0,T]\times \partial\Omega$.

The coefficient functions $a$ and $b$ of main interest are of the form
\begin{equation}\label{eq:CK}
	a(\poro) = a_0\poro^n, \qquad b(\poro) = b_0\poro^m,
\end{equation}
with real constants $a_0, b_0 > 0$ and $n, m \geq 1$. This assumption on $a$ is motivated by the Carman-Kozeny relationship \cite{Costa2006} between the porosity $\poro$ and the permeability of the medium. The function $\sigma$ accounts for \emph{decompaction weakening} \cite{Raess2018,Raess2019} and $\sigma/\poro^m$ can be regarded as the effective viscosity. 

For modeling sharp transitions between materials, it is important to be able to treat porosities with \emph{jump discontinuities}. These turn out to be determined mainly by the initial datum $\poro_0$ for the porosity.
As shown in \cite{Bachmayr2023}, under appropriate conditions on $\poro_0$ that permit jump discontinuities,  these generally remain present also in the corresponding solution $\phi$, but under the given model cannot change their spatial location.

\subsection{Existing numerical methods and novelty}
Many different methods have been proposed to solve the above type of problem numerically, for example finite difference schemes with implicit time-stepping in~\cite{Connolly1998} and adaptive wavelets in~\cite{Vasilyev1998}. In a number of recent works, pseudo-transient schemes based on explicit time stepping in a pseudo-time variable have been investigated. Due to their compact stencils, low communication overhead and simple implementation, such schemes are well suited for parallel computing on GPUs, so that very high grid resolutions can be achieved to compensate the low order of convergence, as shown for example in \cite{Raess2014,Raess2018,Raess2019,Reuber2020,Yarushina2020,Utkin2021}.
Even though all of these schemes are observed to work well for smooth initial porosities $\poro_0$, their convergence can be very slow in problems with nonsmooth $\poro_0$, in particular in the presence of discontinuities. In such cases, due to the smoothing that is implicit in the finite difference schemes, accurately resolving sharp localized features can require extremely large grids. An example is shown in \cref{fig:comparison}\REV{; for details on the model setup, see \cref{sec:setup}.}
\begin{figure}[h!]
	\centering
	\vspace*{-0.3cm}
	\capstart
	\includegraphics[width=\convfac\textwidth]{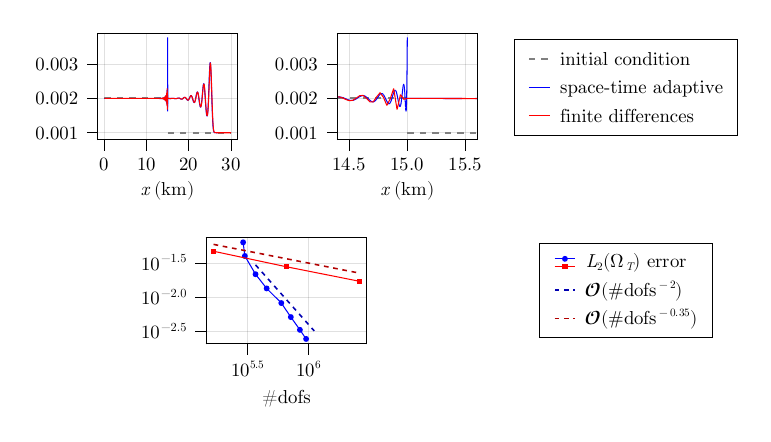}
	\vspace*{-0.8cm}
	\caption{Porosity approximation of space-time adaptive solver with polynomial degree $3$ and finite difference solver (top left) with zoom-in at the discontinuity (top right) and associated convergence rates (bottom; the space-time approximation by the finite difference scheme in the comparison uses a grid with sizes $\Delta t \eqsim \Delta x$, but is computed with smaller intermediate time steps for stability).}\label{fig:comparison}
\end{figure}%

We introduce a space-time adaptive method for solving~\eqref{eq:model} based on a combination of Picard iteration for \eqref{eq:modelphieq} and a particular adaptive least squares discretization of \eqref{eq:modelueq}. While we focus on this particular model case, the approach can be generalized, for example, to similar problems with full force balance, where~\eqref{eq:modelueq} is replaced by a time-dependent Stokes problem as in \cite{Raess2019}.

The adaptive scheme yields efficient approximations of localized features of solutions, in particular in the presence of discontinuities, and can generate space-time grids corresponding to spatially adapted time steps.
The method provides a posteriori estimates of the error with respect to the exact solution of the coupled nonlinear system of PDEs.
Moreover, we numerically observe optimal convergence rates of the generated discretizations with respect to the total number of degrees of freedom.

\subsection{Outline}
In \cref{sec:modeldescription} we describe the basic equations, as well as some possible simplifications and reformulations. We then consider a space-time method for the parabolic equation in \cref{sec:parabolic} and for the pointwise ODE in \cref{sec:fixedpointdiscr}, which yields a method for the full coupled problem. The convergence of this method is shown in \cref{sec:convergence}, and the resulting adaptively controlled scheme is described in \cref{sec:adapt_tol}. In \cref{sec:applications} we show numerical results (especially in the case of discontinuities) and in \cref{sec:error}  we numerically investigate the convergence rates of the fully adaptive methods from \cref{sec:adapt_tol}. At the end we briefly discuss a similar numerical method for the simplified viscous limit model in \cref{sec:viscous}.

\section{Assumptions and simplified models}\label{sec:modeldescription}
In this section we describe the small-porosity approximation as a common simplification and show numerically that it may not be suitable in the case of initial data of low regularity. Based on a transformed version of the general model that facilitates its analysis with non-smooth data, we then state the mild-weak formulation of~\eqref{eq:model} on which our numerical scheme is based.

We start with a crucial assumption on $\sigma$ required for the analysis of~\eqref{eq:model} in~\cite{Bachmayr2023}, which we also rely on  in what follows.
\begin{assumption}\label{ass:sigma}
	We assume that $\sigma\in C^1(\R)$ satisfies
	\begin{align*}
		\sup_{v \in \R} \sigma(v) < \infty,\quad
		\inf_{v\in \R} \sigma(v) > 0, \quad
		\sigma' \geq 0 \text{ on $\R$},
	\end{align*}
	as well as
	\begin{align*}
		\inf_{v\in \R} \left\{\frac{1}{\sigma(v)}-\frac{v\sigma'(v)}{\sigma^2(v)}\right\}>0,\quad
		c_L=\sup_{v\in \R} \left\{\frac{1}{\sigma(v)}-\frac{v\sigma'(v)}{\sigma^2(v)}\right\}<\infty.
	\end{align*}
\end{assumption}
A trivial example for $\sigma$ fulfilling \cref{ass:sigma} is given by $\sigma(v) = c_0$ for all $v\in\R$ with a constant $c_0>0$, proposed in \cite{Vasilyev1998}. Another example, suggested in \cite{Raess2018,Raess2019} and verified to satisfy \cref{ass:sigma} in \cite{Bachmayr2023}, is 
\begin{align}\label{eq:sigma}
	\sigma(v)=c_0 \bl 1 - c_1 \bl 1 + \tanh \bl -\frac{v}{c_2} \br \br \br ,\quad v\in \R,
\end{align}
which provides a phenomenological model for decompaction weakening.
Here $c_0>0$ is a positive constant, $c_1\in [0,\frac12)$ and $c_2>0$, where $1 + \tanh$ can be regarded as a smooth approximation of a step function taking values in the interval $(0,2)$. 
In \REV{the most} well-studied case $c_1=0$, as considered in \cite{Vasilyev1998}, one observes the formation of porosity waves, whereas $c_1>0$ with appropriate problem parameters and initial conditions can lead to the formation of channels. 
In what follows, it will be convenient to write
\begin{equation}\label{eq:kappa}
	\kappa(v) = \frac{v}{\sigma(v)}\,.
\end{equation}
Note that $\kappa$ is Lipschitz continuous with Lipschitz constant $c_L$ by \cref{ass:sigma}.

\subsection{Small-porosity approximation}
For initial data with $\poro_0(x)\in(0,1]$ for $x \in \Omega$ and bounded $\press$, for a classical solution to~\eqref{eq:model} one has $\poro\leq 1$ due to the presence of the factor $(1-\poro)$ in~\eqref{eq:modelphieq}. The \emph{small-porosity approximation} consists in replacing the factor $(1-\poro)$ in~\eqref{eq:model} by 1, which gives the simplified model%
\begin{subequations}\label{eq:modelmod}
	\begin{align}
		\partial_t \poro&=- \bigl( {b(\poro)}{\kappa(\press)} +Q\partial_t \press \bigr) ,\label{eq:modelmodphi}\\
		\partial_t \press&=\frac{1}{Q} \bigl( \nabla \cdot a(\poro)(\nabla \press +f)-{b(\poro)}{\kappa(\press)} \bigr) \label{eq:modelmodu}.
	\end{align}
\end{subequations}
We consider~\eqref{eq:modelmod} subject to the same boundary conditions on $\press$ and initial data for $\poro$ and $\press$ as for \eqref{eq:model}.

For small $\poro$, it is typically assumed that the qualitative behavior of solutions to~\eqref{eq:modelmod} are similar to the ones of the original model~\eqref{eq:model}.
However, the small-porosity approximation can lead to unphysical solutions in the case of a discontinuous $\phi_0$. This can be seen on the left plot in \cref{fig:unphysical_physical}, where starting from typical porosity values of at most $0.2$, the solution develops a peak where $\phi > 1$ at the location of the discontinuity. \REV{See \cref{sec:setup} for details on the model parameters.} Hence we are interested in keeping the factor $(1-\poro)$ in what follows.
\begin{figure}[h!]
	\centering
	\vspace*{-0.3cm}
	\capstart
	\FPmul\result{0.95}{\convfac}
	\includegraphics[width=\result\textwidth]{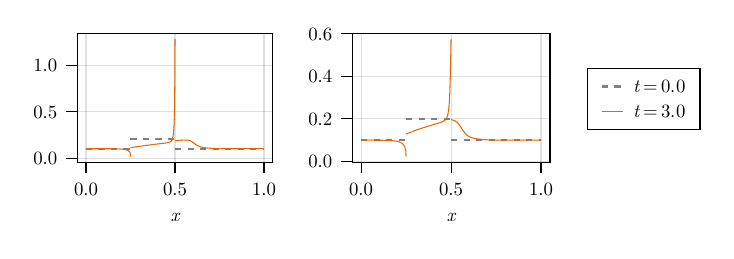}
	\vspace*{-0.6cm}
	\caption{Unphysical solution behavior of the porosity due to the low-porosity approximation (left) and physically correct behavior of the transformed problem (right).}\label{fig:unphysical_physical}
\end{figure}%

This leads to another difficulty, namely that for the full coupled problem~\eqref{eq:model} with non-smooth initial porosity $\poro_0$, the interpretation of the first equation~\eqref{eq:modelphieq} is not obvious, since it contains a term of the form $(1-\poro)\partial_t \press$. When $\poro$ has jump discontinuities in the spatial variables, $\partial_t \press$ in general only exists in the distributional sense, that is, as an element of $L_2(0,T; H^{-1}(\Omega))$. In this case, the product of the distribution $\partial_t \press$ and  $(1-\poro)$, which is  not weakly differentiable, may not be defined.
However, the original problem~\eqref{eq:model} including the factor $(1-\poro)$ can be reduced to a similar form as~\eqref{eq:modelmod} by the following observation: \eqref{eq:modelphieq} can formally be rewritten as
\begin{align*}
	\partial_t \log (1 - \poro) = {b(\poro)}{\kappa(\press)} +Q\partial_t \press .
\end{align*}
Introducing the new variable $\poroTransform = - \log(1-\poro)$, so that $\poro = 1 - e^{-\poroTransform}$, the system~\eqref{eq:model} can be written in the form
\begin{subequations}\label{eq:logtransformed}
	\begin{align}
		\partial_t \poroTransform &= - \bl {b(1 - e^{-\poroTransform})}{\kappa(\press)} +Q\partial_t \press \br , \\
		\partial_t \press &= \frac{1}{Q} \bl \nabla \cdot a(1-e^{-\poroTransform})(\nabla \press + e^{-\poroTransform} f)-{b(1 - e^{-\poroTransform})}{\kappa(\press)} \br ,
	\end{align}
\end{subequations}
which has the same structure as~\eqref{eq:modelmod}. Physically meaningful solutions with $0<\poro <1$ are obtained precisely when $\poroTransform > 0$. As we shall see, the reformulation~\eqref{eq:logtransformed} is also advantageous for obtaining a weak formulation, and we will thus consider~\eqref{eq:model} in this form.

Using this transformation, the unphysical behavior shown in Figure \ref{fig:unphysical_physical} can be prevented without changing the general numerical method. The right plot in \cref{fig:unphysical_physical} shows the solution of the transformed problem~\eqref{eq:logtransformed} for the same parameters and initial setup, with $\phi < 1$ as expected.
This shows that it is in general favorable to consider the full model instead of the low-porosity approximation, especially since it does not require more computational effort to solve the transformed problem~\eqref{eq:logtransformed}.

\subsection{Mild and weak formulations}\label{sec:finalform}
Next we introduce the basic notions of solutions for the different formulations of the problem that we consider in the following sections. The viscoelastic models~\eqref{eq:modelmod} and~\eqref{eq:logtransformed} are both of the general form
\begin{subequations}\label{eq:modelgeneral}
	\begin{align}
		\partial_t \poroGeneral &= -{\beta(\poroGeneral)}{\kappa(\press)} - Q \partial_t \press ,\label{eq:modelgeneral_a}\\
		\partial_t \press &= \frac{1}{Q} \bl \nabla \cdot \alpha(\poroGeneral) \bigl(\nabla \press + \zeta(\poroGeneral)\bigr)-{\beta(\poroGeneral)}{\kappa(\press)} \br ,\label{eq:modelgeneral_b}
	\end{align}
\end{subequations}
where $\alpha, \beta$ and $\zeta$ are given locally Lipschitz continuous functions, with initial conditions $\poroGeneral(0, \cdot) = \poroGeneral_0$ and $\press(0,\cdot) = \press_0$ in $\Omega$.
Note that since $\poroGeneral$ in \eqref{eq:modelgeneral} is in general bounded from above and below, on this range the functions $\alpha, \beta$ and $\zeta$ satisfy a uniform Lipschitz condition. 

To give a meaning to these equations for data of low regularity (in particular, when only $\poroGeneral_0 \in L_\infty(\Omega)$ is assumed), 
we write \eqref{eq:modelgeneral_a} in integral form and consider a weak formulation of \eqref{eq:modelgeneral_b}. This leads us to the formulation, for a.e.\ $t \in [0,T]$,
\begin{subequations}\label{eq:mildweak}
	\begin{align}
		\poroGeneral(t,\cdot) &= \poroGeneral_0 + Q\press_0 - Q\press(t,\cdot) - \int_0^t \beta(\poroGeneral\bigl(s,\cdot)\bigr) \, \kappa\bigl(\press(s,\cdot)\bigr) \D s  &  & \text{in $L_2(\Omega)$,}  \label{eq:mild}\\
		\partial_t \press &= \frac{1}{Q} \bl \nabla \cdot \alpha(\poroGeneral) \bigl(\nabla \press + \zeta(\poroGeneral)\bigr)-\beta(\poroGeneral) \, \kappa(\press) \br & & \text{in $H^{-1}(\Omega)$,} \label{eq:weak}
	\end{align}
\end{subequations}
subject to Dirichlet boundary conditions for $\press$ and initial data $\poroGeneral(0,\cdot)=\poroGeneral_0$, $\press(0,\cdot)=\press_0$ in $\Omega$ for some given $\poroGeneral_0, \press_0 \in L_2(\Omega)$.

\REV{
\subsection{Well-posedness}\label{sec:well-posedness}
In addition to \cref{ass:sigma} on $\sigma$ (and hence, in view of \eqref{eq:kappa}, on $\kappa$), we make the following assumptions on $\alpha$, $\beta$ and $\zeta$ to obtain well-posedness of solutions and convergence of the numerical method.
\begin{assumption}\label{ass:alphabeta}
	We assume that $\alpha, \beta \in C^{0,1}_\mathrm{loc}(\R^+)$, $\zeta \in C^{0,1}_\mathrm{loc}(\R^+;\R^d)$
	and that $\alpha$ is strictly positive on $\R^+$; in other words, for each $\delta > 0$ there exists an $\epsilon > 0$ such that for all $x \in [\delta,\infty)$ we have $\alpha(x) \geq \epsilon > 0$.
	Furthermore we assume that $\beta(x) \geq 0$ for each $x \in \R^+$.
\end{assumption}
\begin{assumption}\label{ass:parabolic}
	We assume that $\sup_{t\in[0,T]} \norm{\poroGeneral(t,\cdot)}_{L_{\infty}(\Omega)}\leq R$, $\inf_{t\in[0,T]} \poroGeneral(t,x) \geq \epsilon$ for a.e.\ $x \in \Omega$ as in~\cite{Bachmayr2023}. Furthermore, let
	\begin{align}\label{eq:W1inf}
		\bignorm{\nablax\paraOp[\poroGeneral]}_{L_{\infty}(\Omega_T)} \leq \overline{C} \,,
	\end{align}
	with $\overline{C}$ depending only on $R$ and $\epsilon$.
\end{assumption}
Under \cref{ass:alphabeta,ass:parabolic,ass:sigma}, we obtain that~\eqref{eq:mildweak} has a unique solution
\begin{equation}
\label{eq:uniquesol}
	(\poroGeneral,\press) \in L_{\infty}(\Omega_T) \times \Big(L_2(0,T;H^1(\Omega))\cap H^1(0,T;H^{-1}(\Omega))\Big).
\end{equation}
This is an immediate consequence of the contraction property shown in \cref{lem:poroContraction} below.  
The proof can be carried out as in \cite[Thm.~4.6]{Bachmayr2023}, where a similar result was shown without \cref{ass:parabolic}. However, the different norms in \eqref{eq:uniquesol} are required for the analysis of our numerical scheme.
}

\REV{
While \cref{ass:alphabeta,ass:sigma} are generally true for the model cases we consider here, \cref{ass:parabolic} does not directly follow from regularity theory. It can be shown, however, for data with piecewise smooth data under additional restrictions on the type of jump discontinuity in $\poroGeneral_0$, summarized in the following remark.
\begin{remark}\label{rem:piecewise}
	Assume $\Omega^j\subset \Omega$ for $j = 1,\ldots,M$ being pairwise disjoint open subsets such that $\overline{\Omega} = \bigcup_{j=1}^M \overline{\Omega}^j$, $\Omega^j \Subset \Omega$ for $j = 1,\ldots,M-1$, $\partial \Omega \subset \partial \Omega^M$ and $\Omega^j$ has a $C^{1,\mu}$-boundary with $\mu>0$. If $\poroGeneral_0 \in C^{0,\gamma}(\overline{\Omega}^j)$ and  $\press_0 \in C^{1,\gamma}(\overline{\Omega}^j)$ for $j=1,\ldots,M$, $\gamma \in (0,\mu/(1+\mu)]$ we get a solution $(\poroGeneral,\press) \in C^{0,\gamma}_\mathrm{par}(\overline{\Omega}_T^j) \times C^{1,\gamma}_\mathrm{par}(\overline{\Omega}_T^j)$  by \cite[Thm.~4.6]{Bachmayr2023}.
	This immediately implies
	\begin{align*}
		\bignorm{\nablax \paraOp[\poroGeneral]}_{L_\infty(\Omega_T)} \leq \overline{C},
	\end{align*}
	with a constant independent of $\poroGeneral$ due to the uniform boundedness of $\poroGeneral$.
\end{remark}
Note that \Cref{rem:piecewise} does not cover all numerically relevant cases for $d \geq 2$ due to the smoothness assumptions on the subdomain boundaries $\partial\Omega^j$.
The spaces $C^{k,\gamma}_\mathrm{par}(\overline{\Omega}_T^j)$, $k=0,1$ are defined in \cref{sec:holder}.
}

\section{Inexact fixed-point iteration}\label{sec:fixedpoint}
Similar to the well-posedness results in \cite{Bachmayr2023}, we perform a Picard iteration for $\poroGeneral$  in order to solve~\eqref{eq:mild}.
We denote the solution for $\press$ given a fixed $\overline{\poroGeneral}$ by $\paraOp[\overline{\poroGeneral}]$. The \REV{$(k+1)$-th iterate} then reads
\begin{align}\label{eq:fixedpoint_poro}
	\poroGeneral^{(k+1)}(t,\cdot)
	= \poroGeneral_0 - Q (\paraOp[\poroGeneral^{(k)}](t,\cdot) - \press_0) - \int_0^t \beta(\poroGeneral^{(k)}(s,\cdot)) \kappa(\paraOp[\poroGeneral^{(k)}](s,\cdot)) \D s.
\end{align}
One may iterate until reaching a certain tolerance determined, for example, by an a posteriori error estimate based on contractivity. As shown in \cref{sec:lipschitz} \REV{(see also \cite[Sec.~4]{Bachmayr2023})}, the mapping defined by the right-hand side of \eqref{eq:fixedpoint_poro} is indeed a contraction \REV{with respect to appropriate norms} for sufficiently small $T$.
In the following section, we consider a numerical scheme for~\eqref{eq:weak} for given $\overline{\poroGeneral}$. We then turn to the discretization of~\eqref{eq:fixedpoint_poro} in \cref{sec:fixedpointdiscr}.

\subsection{Treatment of the parabolic equation}\label{sec:parabolic}
To solve~\eqref{eq:weak} numerically for a given $\overline{\poroGeneral}$, we linearize it by means of another Picard iteration, which leads to solving
\begin{align}\label{eq:fixedpoint_press}
	\partial_t \press^{(k)} = \frac{1}{Q} \bl \nabla \cdot \alpha(\overline{\poroGeneral}) (\nabla \press^{(k)} + \zeta(\overline{\poroGeneral}))-\beta(\overline{\poroGeneral}) \, \frac{\press^{(k)}}{\sigma(\press^{(k-1)})}\br, \quad \press^{(k)}(0,\cdot) = \press_0,
\end{align}
given the previous iterate $\press^{(k-1)}$. We start with an initial iterate $\press^{(0)}$ which, unless stated otherwise, will be a constant continuation of $\press_0$.
Following \cite{Gantner2021,Gantner2024}, let
\begin{align*}
	\pressSolSp
	= \left\{ (\press,\pressFlux) \in L_2(0,T;H_0^1(\Omega)) \times L_2(\Omega_T)^d \,\colon\, \Div(\press,\pressFlux) \in L_2(\Omega_T) \right\},
\end{align*}
with the induced graph norm
\begin{align}\label{eq:graphnorm}
	\norm{(\press,\pressFlux)}_\pressSolSp^2
	= \norm{(\press,\pressFlux)}_{L_2(\Omega_T,\R^{d+1})}^2 + \norm{\nablax \press}_{L_2(\Omega_T,\R^d)}^2 + \norm{\Div(\press,\pressFlux)}_{L_2(\Omega_T)}^2,
\end{align}
where $\Div(\press,\pressFlux) = \partial_t \press + \Div_x\pressFlux$ denotes the space-time divergence. Moreover, let
\begin{align*}
	\pressRanSp
	= L_2(\Omega_T) \times L_2(\Omega_T,\R^d) \times L_2(\Omega),
\end{align*}
endowed with its canonical norm, and 
\begin{align}\label{eq:operator}
	\pressLinOp[\overline{\press}](\press,\pressFlux) = \begin{pmatrix} \Div(\press, \pressFlux) + \Qbeta \, \frac{\press}{\sigma(\overline{\press})}\\\pressFlux + \Qalpha \, \nablax \press\\\press(0,\cdot)\end{pmatrix},\qquad
	\pressLinRhs = \begin{pmatrix}0\\- \Qalpha \, \zeta\\\press_0\end{pmatrix},
\end{align}
where we absorbed $\overline{\poroGeneral}$ and $\frac{1}{Q}$ into the coefficients $\Qalpha, \Qbeta \in L_\infty\REV{(\Omega_T)}$. This allows us to rewrite~\eqref{eq:fixedpoint_press} as
\begin{align}\label{eq:systemlin}
	\pressLinOp\big[\press^{(k-1)}\big]\big(\press^{(k)},\pressFlux^{(k)}\big) = \pressLinRhs,
\end{align}
similar to \cite{Fuehrer2021,Gantner2021,Gantner2024}.
Now \cite[Thm.~2.3]{Gantner2021} yields the following result on the well-posedness of~\eqref{eq:systemlin}.
\begin{theorem}\label{thm:isomorphism}
	Let $\overline{\press} \in U$ and $\Qalpha, \Qbeta \in L_\infty(\Omega_T)$ with $\Qalpha$ uniformly positive. Then $\pressLinOp[\overline{\press}]: \pressSolSp \to \pressRanSp$ is an isomorphism.
\end{theorem}

Due to \cref{ass:sigma,ass:alphabeta}, the assumptions of \cref{thm:isomorphism} are fulfilled. Furthermore, the norm induced by $\pressLinOp[\overline{\press}]$ is equivalent to $\norm{\cdot}_\pressSolSp$ independently of $\overline{\press}$ due to the uniform boundedness of $\sigma$ from above and below. 

Similar to~\cite{Gantner2024}, we discretize $\pressSolSp$ by partitioning $\Omega$ and $(0\,,T)$ separately, which leads to a partition $\mathcal{T}_{\press}$ of prisms. In this work we focus on the case of \REV{rectangular} elements to discretize $\Omega$, which leads to the definition of ($d+1$)-dimensional \REV{rectangles} $\textbf{I}=I_1\times\ldots\times I_{d+1}\in\mathcal{T}_{\press}$ where \REV{$I_{1}$} denotes the temporal direction. \REV{Note that this is convenient for the applications that we target here, but more general prismatic partitions as in \cite{Gantner2024} could be used in the same manner.} We write $\textbf{I}_x = \REV{I_2\times\ldots\times I_{d+1}}$ and define local shape functions
\begin{align*}
	\shapeFun_{\REV{q,p}}(\textbf{I})
	= \big( \polySp_{\REV{q}+1}(\REV{I_{1}}) \otimes \polySpMult_{\REV{p}}(\textbf{I}_x)  \big) \times \big( \polySp_{\REV{q}}(\REV{I_{1}}) \otimes \RT_{\REV{p}}(\textbf{I}_x) \big),
\end{align*}
on $\textbf{I}\in\mathcal{T}_{\press}$ as in \cite[Sec.~2]{Gantner2024} where $\polySp_{\REV{p}}(I)$, $I\subseteq \R$ denotes polynomials of degree $\REV{p}$ and
\begin{align}\label{eq:polyspace}
	\begin{split}
		\polySpMult_{\REV{p}_1,\ldots,\REV{p}_d}(\textbf{I}_x) &= \polySp_{\REV{p}_1}(I_1) \otimes \ldots \otimes \polySp_{\REV{p}_d}(I_d),\\
		\polySpMult_{\REV{p}} &= \polySpMult_{\REV{p},\ldots,\REV{p}}(\textbf{I}_x),\\
		\RT_{\REV{p}}(\textbf{I}_x) &= \polySpMult_{\REV{p}+1,\REV{p},\ldots,\REV{p}}(\textbf{I}_x) \times \ldots \times \polySpMult_{\REV{p},\ldots,\REV{p},\REV{p}+1}(\textbf{I}_x).
	\end{split}
\end{align}
Then we consider the conforming subspace
\begin{equation}\label{eq:fespace}
	\pressSolSpDiscr(\mathcal{T}_{\press})
	= \big\{ (\pressDiscr,\pressFluxDiscr) \in H^1(0,T; H^1_0(\Omega)) \times L_2(0,T;H_{\Div_x}(\Omega)) \,\colon (\pressDiscr,\pressFluxDiscr)|_{\textbf{I}} \in \shapeFun_{\REV{q,p}}(\textbf{I}), \textbf{I} \in \mathcal{T}_{\press} \big\}.
\end{equation}
In the case of axis-parallel cubes with trivial normal vectors, the conformity corresponds to $\pressDiscr$ being continuous and $(\pressFluxDiscr)_i$ being continuous in the $i$-th spatial direction.
As proposed in~\cite[Sec.~2]{Gantner2024}, we will  restrict ourselves to the optimal polynomial degrees $\REV{q}+1=\REV{p}$ in order to achieve better convergence rates.

We solve~\eqref{eq:systemlin} numerically for fixed $\overline{\press}$ in the least-squares formulation of~\cite{Fuehrer2021,Gantner2021,Gantner2024}, which adapted to the present case, in terms of the definitions in \eqref{eq:operator}, reads
\begin{align}\label{eq:leastsquares}
	(\pressDiscr,\pressFluxDiscr)
	= \argmin_{(\pressAltDiscr,\pressFluxAltDiscr) \in \pressSolSpDiscr} \norm{\pressLinOp[\overline{\press}](\pressAltDiscr,\pressFluxAltDiscr) - \pressLinRhs}_\pressRanSp \,.
\end{align}
With the associated bilinear form $\pressBilin$ and right-hand side $\pressRhs$ given by
\begin{align*}
	\pressBilin((\pressDiscr,\pressFluxDiscr),(\pressAltDiscr,\pressFluxAltDiscr))
	= \ip{\pressLinOp[\overline{\press}](\pressDiscr,\pressFluxDiscr)}{\pressLinOp[\overline{\press}](\pressAltDiscr,\pressFluxAltDiscr)}_\pressRanSp,  \quad
	\pressRhs(\pressAltDiscr,\pressFluxAltDiscr)  
	= \ip{\pressLinRhs}{\pressLinOp[\overline{\press}](\pressAltDiscr,\pressFluxAltDiscr)}_\pressRanSp \,,
\end{align*}
the solution $(\pressDiscr,\pressFluxDiscr) \in \pressSolSpDiscr$ of \eqref{eq:leastsquares} is characterized by
\begin{align*}
	\pressBilin((\pressDiscr,\pressFluxDiscr),(\pressAltDiscr,\pressFluxAltDiscr)) = \pressRhs(\pressAltDiscr,\pressFluxAltDiscr)
	\quad \text{for all $(\pressAltDiscr,\pressFluxAltDiscr) \in \pressSolSpDiscr$.}
\end{align*}
Since the residual is evaluated in the $L_2$-space $V$, its $L_2$-norms on elements of $\mathcal{T}_{\press}$ yield reliable and computable local error estimators that can be used to drive an adaptive refinement routine. By \cref{thm:isomorphism}, we furthermore have an equivalence between error and residual,
\begin{align}\label{eq:errestimate}
	\norm{(\press,\pressFlux) - (\pressDiscr,\pressFluxDiscr)}_\pressSolSp
	\eqsim \norm{\pressLinOp[\overline{\press}](\pressDiscr,\pressFluxDiscr) - \pressLinRhs}_\pressRanSp \,,
\end{align}
for $\norm{\cdot}_\pressSolSp$ defined in~\eqref{eq:graphnorm}.
\begin{remark}\label{rem:taylor}
	It is possible to linearize~\eqref{eq:weak} differently by performing a linearization of the term $\frac{\press}{\sigma(\press)}$ in $\overline{\press}$, which yields
	\begin{align}\label{eq:taylor}
		\partial_t \press = \nabla \cdot \Qalpha(\overline{\poroGeneral}) (\nabla \press + \zeta(\overline{\poroGeneral}))-\Qbeta(\overline{\poroGeneral}) \bl \frac{\press}{\sigma(\overline{\press})} - \frac{\overline{\press} \, \sigma'(\overline{\press})}{\sigma(\overline{\press})^2} (\press - \overline{\press})\br,\quad\press(0,\cdot) = \press_0.
	\end{align}
	The resulting Gau\REV{ss}-Newton-type iteration generally converges faster than the simpler quasilinear iteration in \eqref{eq:fixedpoint_press}. 
\end{remark}
\REV{Next} we introduce the discrete parabolic solution operator $\paraOpDiscr$ which is used in the subsequent sections.
\begin{definition}\label{def:paraOp}
	We define $\paraOpDiscr[\overline{\poroGeneral},\overline{\press}] = (\pressDiscr,\pressFluxDiscr)$ to be the solution of~\eqref{eq:leastsquares}
	\REV{such that
		\begin{align}\label{eq:lsqlin}
			\bignorm{G[\overline{\press}](\pressDiscr^{(\ell)},\pressFluxDiscr^{(\ell)}) - R}_\pressRanSp \leq \tol_\mathrm{lsq}
		\end{align}
		with}
	a tolerance $\tol_\mathrm{lsq} > 0$ and set
	\begin{align}\label{eq:lsqnonlin}
		\paraOpDiscr[\overline{\poroGeneral}] = (\pressDiscr^{(\ell)},\pressFluxDiscr^{(\ell)}) = \paraOpDiscr[\overline{\poroGeneral},\pressDiscr^{(\ell-1)}],
	\end{align}
	such that $\bignorm{G[\pressDiscr^{(\ell)}](\pressDiscr^{(\ell)},\pressFluxDiscr^{(\ell)}) - R}_\pressRanSp \leq \tol_\press$ holds for some given tolerance $\tol_\press \geq \tol_\mathrm{lsq} > 0$.
\end{definition}

\subsection{A space-time adaptive fixed-point method}\label{sec:fixedpointdiscr}
\REV{Given a partition of $\mathcal{T}_{\poroGeneral}$ of $\Omega_T$ into prisms, where $\polySpMult_{m}$ denotes the space-time tensor polynomial space defined in~\eqref{eq:polyspace}, we define the discontinuous finite element space}
\begin{align}\label{eq:poroSolSpD}
	\REV{ X_\delta^{m}(\mathcal{T}_{\poroGeneral}) = \{ f \in L_2(\Omega_T) \,\colon f|_{\textbf{I}} \in \polySp_{m}(I_{t}) \otimes \polySpMult_{m}(\textbf{I}_{x}) \text{ for each } \textbf{I}\in\mathcal{T}_{\poroGeneral} \} \,. }
\end{align}
To approximate $\poroGeneral$, we aim to discretize~\eqref{eq:fixedpoint_poro} while maintaining convergence of the fixed-point iteration. This can be done using \eqref{eq:poroDiscr}, where we consider a given approximation $\paraOpDiscr[\poroGeneralDiscr^{(k)}]$ of $\paraOp[\poroGeneralDiscr^{(k)}]$ from \cref{def:paraOp} on some adaptively refined space-time grid.
Then we compute
\begin{align}\label{eq:poroDiscr}
	\begin{split}
		\poroGeneralDiscr^{(k+1)}
		=\, \projOp\bigg( \poroGeneral_{0} - Q \Big( \paraOpDiscr[\poroGeneralDiscr^{(k)}](t,\cdot) - \press_{0} \Big)
		 - \int_0^t \interOp\Big(\beta(\poroGeneralDiscr^{(k)}(s,\cdot)) \, \kappa(\paraOpDiscr[\poroGeneralDiscr^{(k)}](s,\cdot))\Big) \D s \bigg),
	\end{split}
\end{align}
where $\interOp$ denotes interpolation with high-order polynomials (using Chebyshev nodes) such that the resulting \REV{interpolation} error is bounded by a given tolerance $\tol_\mathrm{int}>0$.

Next we perform exact integration of the polynomial approximations. Note that it is important for this step to merge the grids for $\pressDiscr$ and $\poroGeneralDiscr^{(k)}$ first and make them uniform in time (that means without hanging nodes on time-facets) such that we can calculate the integral without running into problems with possible discontinuities in space. 

The resulting high-order polynomial on a time-uniform grid is projected to a lower-order polynomial on an adaptive grid \REV{$\mathcal{T}_{\poroGeneral}$} by a projection $\projOp$ such that the \REV{projection} error is bounded by $\tol_\mathrm{proj}>0$.
For this step we use an adaptive $L_2$-projection based on the $h$-adaptive approximation method derived in \cite[Sec.~2]{Binev2018}, which aligns with the theory developed in \cref{sec:lipschitz}. This projection is chosen in a specific way to allow for a temporal decomposition of $\Omega_T$ discussed in \cref{sec:temporal,sec:timeslices}.
\begin{algorithm}[h!]
	\caption{\textsc{Full Method} to solve~\eqref{eq:modelgeneral}}\label{alg:coupled}
	\begin{algorithmic}
		\REQUIRE $\tol_\poroGeneral$, $\tol_\mathrm{proj}$, $\tol_\mathrm{int}$, $\tol_\press$, $\tol_\mathrm{lsq}$, $\press_0$, $\poroGeneral_0$
		\ENSURE $\poroGeneralDiscr$, $\pressDiscr$
		\STATE initialize $\poroGeneralDiscr^{(0)}\REV{(t,x) = \poroGeneral_0(x), \res_\poroGeneral = \tol_\poroGeneral+1, k = 0}$
		\WHILE{$\res_\poroGeneral > \tol_\poroGeneral$}
		\STATE initialize $\pressDiscr^{(0)}\REV{(t,x) = \press_0(x), \res_\press = \tol_\press+1, \ell = 0}$
		\WHILE{$\res_\press > \tol_\press$}
		\STATE solve $\REV{(\pressDiscr^{(\ell+1)},\pressFluxDiscr^{(\ell+1)})} = \paraOpDiscr[\poroGeneralDiscr^{(k)},\pressDiscr^{(\ell)}]$ up to $\tol_\mathrm{lsq}$ \REV{by means of~\eqref{eq:lsqlin}}
		\STATE compute $\res_\press = \bignorm{G[\pressDiscr^{(\ell+1)}](\pressDiscr^{(\ell+1)},\pressFluxDiscr^{(\ell+1)}) - R}_\pressRanSp$
		\STATE \REV{$\ell = \ell+1$}
		\ENDWHILE
		\STATE calculate $\poroGeneralDiscr^{(k+1)}$ by \eqref{eq:poroDiscr} up to $\tol_\mathrm{proj}$, $\tol_\mathrm{int}$
		\STATE \REV{compute $\res_\poroGeneral = \bignorm{\poroGeneralDiscr^{(k+1)}-\poroGeneralDiscr^{(k)}}_{\Tspace\REV{(\Omega_T)}}$}
		\STATE \REV{$k = k+1$}
		\ENDWHILE
	\end{algorithmic}
\end{algorithm}

Next we combine the above steps in an adaptive scheme for the full nonlinear problem~\eqref{eq:mildweak}.
There are different ways of combining the methods from \cref{sec:parabolic,sec:fixedpointdiscr}, but the most reliable one coincides with the theoretical ideas from~\cite{Bachmayr2023} and is summarized in \cref{alg:coupled}. There we solve for $\paraOpDiscr[\poroGeneralDiscr]$ and then update $\poroGeneralDiscr$ until the respective error tolerances are fulfilled where the norm $\norm{\cdot}_{\Tspace\REV{(\Omega_T)}}$ is \REV{defined in \eqref{eq:tracenorm}}. A more involved scheme for controlling these tolerances themselves adaptively is presented in \cref{sec:adapt_tol}.

\subsection{Temporal subdivision}\label{sec:temporal}
To ensure convergence of the nonlinear iterations~\eqref{eq:fixedpoint_poro} and~\eqref{eq:fixedpoint_press}, in general we need to split the space-time cylinder $\Omega_T$ into time slices that can be chosen as large as the Lipschitz constants of both iterations allow. Hence their size only depends on the continuous problem, but is independent of the discretization.
However, to maintain control of overall errors, controlling the trace errors at time slice boundaries is crucial. 

Hence doing this re-approximation simply in the norm of $L_2(\Omega_T)$ is not sufficient for controlling the resulting errors of $\gamma_T \poroGeneralDiscr$ in $L_2(\Omega)$-norm, where $\gamma_t: f \mapsto f(t,\cdot)$ is the associated trace operator for each time $t \in [0,T]$. It clearly is bounded as a mapping from $C([0,T];L_2(\Omega))$ to $L_2(\Omega)$ and as a mapping from $\pressSolSp$ to $L_2(\Omega)$, since for the scalar components of elements of $U$ we can apply \cite[Prop.~2.1]{Gantner2021} and the embedding
\begin{equation}\label{eq:imbedding}
	L_2(0,T;H^1_0(\Omega)) \cap H^1(0,T;H^{-1}(\Omega)) \hookrightarrow C([0,T];L_2(\Omega)).
\end{equation}

We thus modify the re-approximation to explicitly account for errors in the trace at $T$ as follows.
Given a partition of $\mathcal{T}_{\REV{\poroGeneral}}$ of $\Omega_T$ into prisms, where \REV{$X_\delta^{m}(\mathcal{T}_{\REV{\poroGeneral}})$} denotes the \REV{discontinuous finite element} space defined in~\REV{\eqref{eq:poroSolSpD}}. Then the projection $\projOp$ is defined as
\begin{align*}
	\projOp f = \argmin_{f_\delta \in X_\delta^{\REV{m}}(\mathcal{T}_{\REV{\poroGeneral}})} \norm{f - f_\delta}_{\Tspace\REV{(\Omega_T)}},
\end{align*}
where
\begin{align}\label{eq:tracenorm}
	\norm{f}_{\Tspace\REV{(\Omega_T)}}^2
	= \norm{f}_{L_2(\Omega_T)}^2 + \norm{\gamma_T f}_{L_2(\Omega)}^2,
\end{align}
is a norm on $C([0,T];L_2(\Omega))$. Note that $\norm{\cdot}_{\Tspace\REV{(\Omega_T)}}$ is induced by an $L_2$-inner product, and thus it is easy to compute the minimizer in the definition of $\Pi$.
Combining this with the adaptive tree refinement of \cite[Sec.~2]{Binev2018}, we obtain a near-best approximation tree. Together with the estimates derived in \cref{sec:nonlin,sec:lipschitz} we \REV{will be able to} control the terminal nonlinear errors of $\poroGeneral$ and $\press$\REV{; details are given in \cref{sec:timeslices}}.

Rather than splitting the domain into time slices, one could also consider a globally coupled approach by solving for subsets of unknowns while freezing the remaining ones. However, due to the global coupling in time in the discretization of the (generally nonlinear) parabolic problem for $\press$, convergence of iterations constructed in this manner is a delicate question. It becomes easier in cases where the parabolic problem is linear, for example when $\sigma$ is constant, but since this is a strong restriction excluding problems with decompaction weakening that are of main interest to us, we do not pursue this direction further.

\section{Convergence}\label{sec:convergence}
In order to prove the convergence of \cref{alg:coupled}, we start by considering a convergence result for an abstract perturbed fixed-point iteration that we subsequently apply to our method.
We assume that $\generalOp$ is a Lipschitz continuous mapping with Lipschitz constant $\generalOpLipschitz < 1$ with respect to \REV{any} suitable norm on a suitably chosen closed set, which implies that Banach's fixed point theorem yields a unique fixed point $\general$ as well as convergence of the fixed point iteration.
Then we can write a discretized iteration in the form
\begin{equation}\label{eq:abstractiter}
	\generalDiscr[(k+1)] = \generalOp(\generalDiscr[(k)]) + \discrErr{k}\,,
\end{equation}
where $\discrErr{k}$ denotes the discretization error. For the resulting error, we then have
$	\bignorm{\general - \generalDiscr[(k+1)]}
	\leq \generalOpLipschitz \bignorm{\general - \generalDiscr[(k)]} + \bignorm{\discrErr{k}}\,,
$
and thus by induction
\begin{equation}\label{eq:errinductive}
	\bignorm{\general - \generalDiscr[(k+1)]}
	\leq (\generalOpLipschitz)^{k+1} \bignorm{\general - \generalDiscr[(0)]} + \sum_{i=0}^{k} (\generalOpLipschitz)^{k-i} \bignorm{\discrErr{i}} \,.
\end{equation}
Furthermore, it is clear that the perturbed fixed point iteration converges if $\bignorm{\discrErr{i}} \rightarrow 0$ for $i\to \infty$. Concerning the speed of convergence, we have the following estimate.
\begin{lemma}\label{lem:conv}
	If for some $\errorRedFac < 1-\generalOpLipschitz$ and all $k \leq N$,
	\begin{align}\label{eq:errorreduction}
		\bignorm{\discrErr{k}}
		\leq \errorRedFac \, (\generalOpLipschitz + \errorRedFac)^k \, \bignorm{\general - \generalDiscr[(0)]},
	\end{align}
	then $\bignorm{\general - \generalDiscr[(k)]} \leq (\generalOpLipschitz + \errorRedFac)^{k} \bignorm{\general - \generalDiscr[(0)]}$
	for $k \leq N+1$.
\end{lemma}
\begin{proof}
	With~\eqref{eq:errinductive} we obtain
	\begin{align*}
		\bignorm{\general - \generalDiscr[(k+1)]}
		&\leq (\generalOpLipschitz)^{k+1} \bignorm{\general - \generalDiscr[(0)]} + \sum_{i=0}^{k} (\generalOpLipschitz)^{k-i} \bignorm{\discrErr{i}}\\
		&\leq \bigg( (\generalOpLipschitz)^{k+1} + \errorRedFac \sum_{i=0}^{k} (\generalOpLipschitz)^{k-i} (\generalOpLipschitz + \errorRedFac)^i \bigg) \bignorm{\general - \generalDiscr[(0)]}\\
		&= (\generalOpLipschitz + \errorRedFac)^{k+1} \bignorm{\general - \generalDiscr[(0)]},
	\end{align*}
	where in the last step we have used that
	$
		(b-a) \sum_{i = 0}^{k} a^{k-i} \, b^i = b^{k+1} - a^{k+1}
	$
	for $a,b \in \R$ and $k \in \N$.
\end{proof}

In what follows, we apply the above to different contractions $\generalOp$ with correspondingly different mechanisms for ensuring errors $\bignorm{\discrErr{k}}_{\Tspace\REV{(\Omega_T)}}$ below the respective thresholds.

\subsection{Nonlinear least-squares method}\label{sec:nonlin}
First we want to apply the general results about perturbed fixed-point iterations in order to prove convergence of the nonlinear least-squares method presented in \cref{sec:parabolic}.

We now assume a fixed $\overline{\poroGeneral}$ to be given. Let the operator $\pressOp$ be defined, for each given $\overline{\press}$, by $\pressOp(\overline{\press}) = u$, where $u$ is the solution of 
\begin{align*}
	\partial_t \press = \nabla \cdot \Qalpha(\overline{\poroGeneral}) (\nabla \press + \zeta(\overline{\poroGeneral})) - \Qbeta(\overline{\poroGeneral}) \, \frac{\press}{\sigma(\overline{\press})}, \quad \press(0,\cdot) = \press_0\,.
\end{align*}
\begin{lemma}\label{lem:pressContraction}
		For $\Phi$ as defined above, we have
		\begin{align*}
			\norm{\pressOp(\overline{\press}_2) - \pressOp(\overline{\press}_1)}_{\Tspace\REV{(\Omega_T)}}
			\lesssim T^{\frac{1}{2}} \norm{\overline{\press}_2 - \overline{\press}_1}_{L_2(\Omega_T)}\,,
		\end{align*}
		which implies that $\pressOp$ is a contraction with respect to $\norm{\cdot}_{\Tspace\REV{(\Omega_T)}}$ with Lipschitz constant $\pressOpLipschitz < 1$ if $T$ is sufficiently small.
\end{lemma}
\begin{proof}
	For solutions $\press_1$ and $\press_2$ given $\overline{\press}_1$ and $\overline{\press}_2$, respectively, we have
	\begin{align*}
		\partial_t (\press_2 - \press_1)
		&= \nabla \cdot \Qalpha(\overline{\poroGeneral}) \nabla (\press_2 - \press_1) - \Qbeta(\overline{\poroGeneral}) \, \bl\frac{\press_2}{\sigma(\overline{\press}_2)} - \frac{\press_1}{\sigma(\overline{\press}_1)}\br\\
		&= \nabla \cdot \Qalpha(\overline{\poroGeneral}) \nabla (\press_2 - \press_1) - \Qbeta(\overline{\poroGeneral}) \, \bl\frac{\press_2 - \press_1}{\sigma(\overline{\press}_2)} + \press_1 \bl \frac{1}{\sigma(\overline{\press}_2)}-\frac{1}{\sigma(\overline{\press}_1)}\br\br \,.
	\end{align*}
	Using that $\sigma$ is bounded from below and $\press_1$, $\overline{\poroGeneral}$ uniformly from above, \REV{$\frac{1}{\sigma} \in C^{0,1}(\R)$}, we obtain the Lipschitz estimate
	\begin{align}\label{eq:ucontract1}
		\begin{split}
			\norm{\press_2 - \press_1}_{L_2(\Omega_T)}
			&\leq T^{\frac{1}{2}} \norm{\press_2 - \press_1}_{L_\infty(0,T;L_2(\Omega))}\\
			&\lesssim T^{\frac{1}{2}} \Bignorm{\Qbeta(\overline{\poroGeneral}) \, \press_1 \bl \tfrac{1}{\sigma(\overline{\press}_2)}-\tfrac{1}{\sigma(\overline{\press}_1)}\br}_{L_2(\Omega_T)}\\
			&\lesssim T^{\frac{1}{2}} \Bignorm{\tfrac{1}{\sigma(\overline{\press}_2)}-\tfrac{1}{\sigma(\overline{\press}_1)}}_{L_2(\Omega_T)}\\
			&\lesssim T^{\frac{1}{2}} \norm{\overline{\press}_2 - \overline{\press}_1}_{L_2(\Omega_T)}\,,
		\end{split}
	\end{align}
	\REV{where the second inequality follows by standard energy estimates as in \cite[Chap.~7]{Evans2010} (see also \cite[Thm,~4.1]{Bachmayr2023}).}
	Applying \cite[Thm,~4.2]{Bachmayr2023} (based on \cite[\REV{Thm.~V.3.2}]{DiBenedetto1993}) we immediately get
	\begin{align}\label{eq:ucontract2}
		\begin{split}
			\norm{\press_2(T,\cdot) - \press_1(T,\cdot)}_{L_2(\Omega)}
			&\leq |\Omega|^{1/2} \norm{\press_2(T,\cdot) - \press_1(T,\cdot)}_{L_\infty(\Omega)}\\
			&\lesssim \norm{\press_2 - \press_1}_{L_2(\Omega_T)}\\
			&\lesssim T^{\frac{1}{2}} \norm{\overline{\press}_2 - \overline{\press}_1}_{L_2(\Omega_T)},
		\end{split}
	\end{align}
	with constants independent of $T$. Combining~\eqref{eq:ucontract1} and~\eqref{eq:ucontract2} yields
	the desired contraction property.
\end{proof}

Noting that the numerical error of our linear least-squares solver can be made arbitrarily small, the numerical method~\eqref{eq:lsqnonlin} converges if
$
	\bignorm{\discrErr{k}}_{\Tspace\REV{(\Omega_T)}} \leq \tol_\mathrm{lsq}^{(k)} \to 0
$
for $k \to \infty$. Furthermore, error reduction is obtained by a simpler argument than in \cref{lem:conv}, ensuring that
$
	\bignorm{\discrErr{k}}_{\Tspace\REV{(\Omega_T)}}
	\leq \errorRedFac \bignorm{\press - \pressDiscr^{(k)}}_{\Tspace\REV{(\Omega_T)}}
$
holds for some $\errorRedFac < 1-\pressOpLipschitz$. As shown next, this can be guaranteed by considering the \emph{nonlinear residual} $\norm{\pressLinOp[\press](\press,\pressFlux)-\pressLinOp[\pressDiscr](\pressDiscr,\pressFluxDiscr)}_\pressRanSp$, which can be used as an error estimator.

\begin{proposition}\label{prp:nonlinerr}
	If $(\pressDiscr,\pressFluxDiscr)$ is a solution of~\eqref{eq:leastsquares}, then
	\begin{align}\label{eq:errestimatenonlin}
		\norm{(\press,\pressFlux) - (\pressDiscr,\pressFluxDiscr)}_\pressSolSp
		\eqsim \norm{\pressLinOp[\press](\press,\pressFlux)-\pressLinOp[\pressDiscr](\pressDiscr,\pressFluxDiscr)}_\pressRanSp.
	\end{align}
\end{proposition}
\begin{proof}
We calculate the Fr\'{e}chet derivative of the nonlinear operator $\pressLinOp[\press](\press,\pressFlux)$. For $h = (h_1,h_2) \in \pressSolSp$ this yields
\begin{equation*}
	\pressNonLinOpDiff[\press] h
	= \begin{pmatrix}\Div(h_1,h_2) + \Qbeta \frac{\sigma(\press) - \press\sigma'(\press)}{\sigma(\press)^2} h_1\\ h_2 + \Qalpha \nablax h_1\\ h_1(0,\cdot)\end{pmatrix}\,,
\end{equation*}
where $\Qbeta \frac{\sigma(\press) - \press\sigma'(\press)}{\sigma(\press)^2}$ is uniformly bounded due to uniform bounds on $\Qbeta$ and \cref{ass:sigma}.

By \cref{thm:isomorphism}, with homogeneous Dirichlet boundary data, $\pressNonLinOpDiff[\press]: \pressSolSp\to\pressRanSp$ is an isomorphism. Hence we have bounds
\begin{align}\label{eq:uniformbounds}
	\norm{\pressNonLinOpDiff[\press]}_{\pressSolSp\to\pressRanSp} \leq C_1, \quad \norm{\pressNonLinOpDiff[\press]^{-1}}_{\pressRanSp\to\pressSolSp} \leq C_{-1},
\end{align}
with $C_1, C_{-1} >0$ independent of $\press$ due to the uniform bounds on $\frac{\sigma(\press) - \press\sigma'(\press)}{\sigma(\press)^2}$.
Thus
\begin{align*}
	\norm{\pressNonLinOp[\press](\press,\pressFlux)-\pressNonLinOp[\pressDiscr](\pressDiscr,\pressFluxDiscr)}_\pressRanSp
	\leq C_1 \norm{(\press,\pressFlux) - (\pressDiscr,\pressFluxDiscr)}_\pressSolSp \,,
\end{align*}
which yields one side of the estimate~\eqref{eq:errestimatenonlin}.

By the bound on $\pressNonLinOpDiff^{-1}$ in~\eqref{eq:uniformbounds}, we can apply the inverse function theorem (see, for example, \cite[Sec.~9.2]{Luenberger1969}) to conclude that for each $(\press,\pressFlux) \in \pressSolSp$ there exists a neighborhood $B$ of $\pressNonLinOp[\press](\press,\pressFlux)$ and a unique Fr\'{e}chet differentiable function $\pressNonLinOpInv:B\to\pressSolSp$ such that
$
	\pressNonLinOp \circ \pressNonLinOpInv(x) = x
$ and  
$	\pressNonLinOpInvDiff(x) = \pressNonLinOpDiff(\pressNonLinOpInv(x))^{-1}$ for all $x \in B$.
As a consequence of~\eqref{eq:uniformbounds}, we have $\norm{\pressNonLinOpInvDiff(x)}_{B\to\pressSolSp} \leq C_{-1}$ for all $x \in B$.
Applying the same argument as before to $\pressNonLinOpInv$, since
$
	\pressNonLinOpInvDiff(x)^{-1} = \pressNonLinOpDiff(\pressNonLinOpInv(x))
$
is bounded by~\eqref{eq:uniformbounds}, we obtain a unique function $\tilde{\pressNonLinOp}: \tilde{B} \to B$ such that $\pressNonLinOpInv\circ\tilde{\pressNonLinOp}(\press,\pressFlux) = (\press,\pressFlux)$ for all $(\press,\pressFlux)$ in a neighborhood $\tilde{B}$ of $\pressNonLinOpInv(x)$ for $x \in B$. Furthermore, for all $(\press,\pressFlux) \in \tilde{B}$,
\begin{align*}
	\pressNonLinOp[\press](\press,\pressFlux)
	= \pressNonLinOp \circ \pressNonLinOpInv \circ \tilde{\pressNonLinOp}(\press,\pressFlux)
	= \tilde{\pressNonLinOp}(\press,\pressFlux) \,.
\end{align*}
In order to prove the converse estimate in~\eqref{eq:errestimatenonlin}, we consider the line segment
\begin{align*}
	[(\press,\pressFlux),(\pressDiscr,\pressFluxDiscr)] = \left\{ \lambda (\press,\pressFlux) + (1-\lambda) (\pressDiscr,\pressFluxDiscr) \,\colon\, \lambda \in [0,1] \right\},
\end{align*}
which is compact in $\pressSolSp$ and hence admits a finite covering of $[(\press,\pressFlux),(\pressDiscr,\pressFluxDiscr)]$ by neighborhoods $\tilde{B}$ where $\pressNonLinOpInv$ and $\tilde{\pressNonLinOp} = \pressNonLinOp$ are defined. We thus obtain a well-defined map $\pressNonLinOpInvDiff$ on the entire line segment,
which in addition is uniformly bounded by $C_{-1}$, so that
\begin{align*}
	\norm{(\press,\pressFlux) - (\pressDiscr,\pressFluxDiscr)}_\pressSolSp
	&= \norm{\pressNonLinOpInv \circ \pressNonLinOp[\press](\press,\pressFlux) - \pressNonLinOpInv \circ \pressNonLinOp[\pressDiscr](\pressDiscr,\pressFluxDiscr)}_\pressSolSp\\
	&\leq C_{-1} \norm{\pressNonLinOp[\press](\press,\pressFlux) - \pressNonLinOp[\pressDiscr](\pressDiscr,\pressFluxDiscr)}_\pressRanSp \,. \qedhere
\end{align*}
\end{proof}

Note that due to the embedding~\eqref{eq:imbedding} this also yields an estimate of $\norm{\press-\pressDiscr^{(\ell)}}_{\Tspace\REV{(\Omega_T)}}$.

\begin{theorem}\label{thm:pressdiscrerr}
	Let $\overline{\poroGeneral}$ be fixed, and with $\pressOpLipschitz$ from \cref{lem:pressContraction}, let $\errorRedFac<1-\pressOpLipschitz$ and $\tol_\mathrm{lsq}$ be chosen such that
	\begin{align*}
		\bignorm{\discrErr{\ell}}_{\Tspace\REV{(\Omega_T)}}
		\lesssim \mathrm{tol}_\mathrm{lsq}
		\lesssim \errorRedFac \, \bignorm{G[\pressDiscr^{(\ell)}](\pressDiscr^{(\ell)},\pressFluxDiscr^{(\ell)}) - R}_\pressRanSp,
	\end{align*}
with the constants from~\eqref{eq:errestimate} and~\eqref{eq:errestimatenonlin} depending on $\norm{\overline{\poroGeneral}}_{L_\infty(\Omega_T)}$ and $\norm{1/\overline{\poroGeneral}}_{L_\infty(\Omega_T)}$. Then \[  \bignorm{\press - \pressDiscr^{(\ell+1)}}_{\Tspace\REV{(\Omega_T)}} \leq (\pressOpLipschitz + \errorRedFac) \bignorm{\press- \pressDiscr^{(\ell)}}_{\Tspace\REV{(\Omega_T)}}\]  \REV{with} $\press = \paraOp[\overline{\poroGeneral}]$ and $\pressDiscr^{(\ell)}$ as in \cref{def:paraOp}.
\end{theorem}

\subsection{Lipschitz estimate}\label{sec:lipschitz}
We now show that the operator $\poroOp$ defined by
\begin{align*}
	\poroOp(\poroGeneral)(t,\cdot)
	= \poroGeneral_{0} - Q \left( \paraOp[\poroGeneral](t,\cdot) - \press_{0} \right) - \int_0^t \beta(\poroGeneral(s,\cdot)) \, \kappa(\paraOp[\poroGeneral](s,\cdot)) \D s
\end{align*}
is a contraction with respect to $\norm{\cdot}_{\Tspace\REV{(\Omega_T)}}$ if $T$ is chosen sufficiently small. 

In \cite[Sec.~4]{Bachmayr2023}, such a property is established with respect to a piecewise $C^{0,\gamma}_\textrm{par}$-norm. 
With respect to the weaker $\Tspace$-norm as used here, contractivity is not known under general assumptions.
\REV{%
However, we obtain the desired estimate by using \cref{ass:parabolic}.
}

\begin{lemma}\label{lem:poroContraction}
	\REV{Let \cref{ass:parabolic,ass:alphabeta,ass:sigma} be satisfied and $\poroGeneral_1, \poroGeneral_2 \in L_{\infty}(\Omega_T)$. Then} it holds that
	\begin{align*}
		\norm{\poroOp(\poroGeneral_2) - \poroOp(\poroGeneral_1)}_{\Tspace\REV{(\Omega_T)}}
		\lesssim T^{\frac{1}{2}} \norm{\poroGeneral_2 - \poroGeneral_1}_{L_2(\Omega_T)}\,,
	\end{align*}
	which implies that $\poroOp$ is a contraction with respect to $\norm{\cdot}_{\Tspace\REV{(\Omega_T)}}$ with Lipschitz constant $\poroOpLipschitz < 1$ if $T$ is small enough.
\end{lemma}
\begin{proof}
	\REV{We write $\press_1 = \paraOp[\poroGeneral_1]$, $\press_2 = \paraOp[\poroGeneral_2]$ for the solutions of~\eqref{eq:weak} for given $\poroGeneral_1, \poroGeneral_2$.}
	Considering a difference equation similar to~\cite[Sec.~4]{Bachmayr2023}, we obtain
	\begin{multline}
		\partial_t (\press_2 - \press_1) - \nabla \cdot (\Qalpha(\poroGeneral_1) \nabla(\press_2 - \press_1)) + \Qbeta(\poroGeneral_1) \Delta_{\kappa, \press_1}(\press_2) (\press_2 - \press_1)\\
		\begin{split}
			=\;& \nabla \cdot \left( (\Qalpha(\poroGeneral_2) - \Qalpha(\poroGeneral_1)) \nabla \press_2 \right) - \kappa(\press_2) (\Qbeta(\poroGeneral_2) - \Qbeta(\poroGeneral_1))\\
			&+ \nabla \cdot \left( \Qalpha(\poroGeneral_2) \zeta(\poroGeneral_2) - \Qalpha(\poroGeneral_1) \zeta(\poroGeneral_1) \right),
		\end{split}
	\end{multline}
	where $\Delta_{\kappa,y}$ is defined as
	\begin{align}\label{eq:Delta}
		\Delta_{\kappa,y}(x) =
		\begin{cases}
			\frac{\kappa(x)-\kappa(y)}{x-y} &\text{if } x \neq y,\\
			\kappa'(y) &\text{else.}
		\end{cases}
	\end{align}
	By standard regularity theory (see, for example, \cite[Thm.~4.1]{Bachmayr2023}), 
	\begin{equation}\label{eq:estpress1}
	\begin{aligned}
		\norm{\press_2 - \press_1}_{L_2(\Omega_T)}
			&\leq T^{1/2} \norm{\press_2 - \press_1}_{L_\infty(0,T;L_2(\Omega))}\\
			&\lesssim T^{1/2} \Big( \norm{(\Qalpha(\poroGeneral_2) - \Qalpha(\poroGeneral_1)) \, \nabla \press_2}_{L_2(\Omega_T)} +  \bignorm{\kappa(\press_2) (\Qbeta(\poroGeneral_2) - \Qbeta(\poroGeneral_1))}_{L_2(\Omega_T)}\\
			&\phantom{\lesssim T^{1/2} \Big(}\; + \norm{\Qalpha(\poroGeneral_2) \zeta(\poroGeneral_2) - \Qalpha(\poroGeneral_1) \zeta(\poroGeneral_1)}_{L_2(\Omega_T)} \Big)\\
			&\lesssim T^{1/2} \Big( \norm{\poroGeneral_2 - \poroGeneral_1}_{L_2(\Omega_T)} \norm{\nabla \press_2}_{L_\infty(\Omega_T)}\\
			&\phantom{\leq T^{1/2} \Big(}\;+ \norm{\poroGeneral_2 - \poroGeneral_1}_{L_2(\Omega_T)} \norm{\kappa(\press_2)}_{L_\infty(\Omega_T)} + \norm{\poroGeneral_2 - \poroGeneral_1}_{L_2(\Omega_T)} \Big)\\
			&\lesssim T^{1/2} \norm{\poroGeneral_2 - \poroGeneral_1}_{L_2(\Omega_T)}.
	\end{aligned}
	\end{equation}
	As before in \cref{lem:pressContraction}, we apply \cite[Thm.~4.2]{Bachmayr2023} \REV{(based on \cite[Thm.~V.3.2]{DiBenedetto1993})} to get
	\begin{align}\label{eq:estpress2}
		\begin{split}
			\norm{\press_2(T,\cdot) - \press_1(T,\cdot)}_{L_2(\Omega)}
			&\leq |\Omega|^{1/2} \norm{\press_2(T,\cdot) - \press_1(T,\cdot)}_{L_\infty(\Omega)}\\
			&\lesssim \norm{\press_2 - \press_1}_{L_2(\Omega_T)}\\
			&\lesssim T^{\frac{1}{2}} \norm{\poroGeneral_2 - \poroGeneral_1}_{L_2(\Omega_T)}\,,
		\end{split}
	\end{align}
	which implies the contraction property with respect to $\norm{\cdot}_{\Tspace\REV{(\Omega_T)}}$.
	To this end, note that
	\begin{multline}\label{eq:estporo}
		\norm{\poroOp(\poroGeneral_2) - \poroOp(\poroGeneral_1)}_{\Tspace\REV{(\Omega_T)}}^2
		\lesssim T \norm{\poroGeneral_2 - \poroGeneral_1}_{L_2(\Omega_T)}^2\\
		\begin{split}
			&+  \biggnorm{\int_0^t \beta(\poroGeneral_2(s,\cdot)) \, \kappa(\paraOp[\poroGeneral_2](s,\cdot)) - \beta(\poroGeneral_1(s,\cdot)) \, \kappa(\paraOp[\poroGeneral_1](s,\cdot)) \D s}_{L_2(\Omega_T)}^2\\
			&+  \biggnorm{\int_0^T \beta(\poroGeneral_2(s,\cdot)) \, \kappa(\paraOp[\poroGeneral_2](s,\cdot)) - \beta(\poroGeneral_1(s,\cdot)) \, \kappa(\paraOp[\poroGeneral_1](s,\cdot)) \D s}_{L_2(\Omega)}^2,
		\end{split}
	\end{multline}
	where we have estimated $\norm{\paraOp[\poroGeneral_1] - \paraOp[\poroGeneral_2]}_{\Tspace\REV{(\Omega_T)}}$ using~\eqref{eq:estpress1} and~\eqref{eq:estpress2}, and we estimate the second term of~\eqref{eq:estporo} by
	\begin{multline*}
		\biggnorm{\int_0^t \beta(\poroGeneral_2(s,\cdot)) \, \kappa(\paraOp[\poroGeneral_2](s,\cdot)) - \beta(\poroGeneral_1(s,\cdot)) \, \kappa(\paraOp[\poroGeneral_1](s,\cdot)) \D s}_{L_2(\Omega_T)}^2\\
		\begin{split}
			&\lesssim \int_{0}^{T} \left( \int_{0}^{t} \norm{\poroGeneral_2(s,\cdot) - \poroGeneral_1(s,\cdot)}_{L_2(\Omega)} \D s \right)^2 \D t\\
			&\leq \int_{0}^{T} t \left( \int_{0}^{t} \norm{\poroGeneral_2(s,\cdot) - \poroGeneral_1(s,\cdot)}_{L_2(\Omega)}^2 \D s \right) \D t\\
			&\leq T^2 \norm{\poroGeneral_2 - \poroGeneral_1}_{L_2(\Omega_T)}^2.
		\end{split}
	\end{multline*}
	A similar argument yields an estimate of the last term of~\eqref{eq:estporo} by $T \norm{\poroGeneral_2 - \poroGeneral_1}_{L_2(\Omega_T)}^2$ which concludes the proof.
\end{proof}
The above contractivity property can be combined well with the approximation $\paraOpDiscr$ of $\paraOp$ as in \cref{def:paraOp}, with error control in matching norms.
\begin{remark}
	Note that an analogous argument also gives
	\begin{align*}
		\norm{\press_2 - \press_1}_{C([0,T];L_2(\Omega))}
		\lesssim T^{1/2} \norm{\poroGeneral_2 - \poroGeneral_1}_{C([0,T];L_2(\Omega))}\,,
	\end{align*}
	which implies convergence of the fixed point iteration for $\poroOp$ in $C([0,T];L_2(\Omega))$. Furthermore the nonlinear least-squares method yields control of the discretization error in that norm because of \cref{prp:nonlinerr}. But since it is more difficult and expensive to perform a re-approximation step such that the $C([0,T];L_2(\Omega))$ error can be controlled, we will not consider it in this work.
\end{remark}

\subsection{Main result}\label{sec:interpolation}
In order to guarantee an error reduction in \cref{alg:coupled} to solve the full viscoelastic model we need to bound the errors $\bignorm{\discrErr{k}}_{\Tspace\REV{(\Omega_T)}}$ in every step of the fixed-point iteration.
We therefore consider the approximated iterate~\eqref{eq:poroDiscr}, which reads
\begin{equation*}
		\poroGeneralDiscr^{(k+1)}
		=\, \projOp\bigg( \poroGeneral_{0} - Q \Big( \paraOpDiscr[\poroGeneralDiscr^{(k)}](t,\cdot) - \press_{0} \Big) - \int_0^t \interOp\Big(\beta(\poroGeneralDiscr^{(k)}(s,\cdot)) \, \kappa(\paraOpDiscr[\poroGeneralDiscr^{(k)}](s,\cdot))\Big) \D s \bigg),
\end{equation*}
where $\paraOpDiscr$ denotes the least-squares solution operator defined in \cref{def:paraOp}, $\projOp$ is the adaptive projection which we introduced in \cref{sec:temporal}, and $\interOp$ is interpolation with high-order polynomials on each element.

In view of the given error tolerances for $\projOp$ and $\paraOpDiscr$ due to \cref{prp:nonlinerr}, it remains to consider the errors of the interpolation $\interOp$. We use the following result from~\cite[Thm.~3.1]{Moessner2009} (see also, e.g., \cite[Sec.~4]{Mason1980}).
\begin{theorem}\label{thm:interpolation}
	\REV{Let $H=[0, h_1]\times\cdots\times[0,h_d] \subset \R^d$ and with $n \in \N_0^d$, let} \[\REV{ f \in \{  g \in L_\infty(H) \colon \partial^{n_i}_{x_i} f \in L_\infty(H), \,i=1,\ldots,d \}. }  \]  	Furthermore, we consider interpolation points $\gamma_r^{\REV{1}} < \ldots < \gamma_r^{n_r}$ with associated La\-gran\-ge basis functions $\REV{\mathcal{L}}_r^{\REV{1}}, \ldots, \REV{\mathcal{L}}_r^{n_r}$ on $[0,h_r]$ for each $r=1,\ldots,d$. Then for the unique interpolant $I[f]$, we have
	\begin{align}\label{eq:interpolation_error}
		\norm{f - I[f]}_{L_\infty(H)}
		\leq \sum_{r=1}^d L_r^d h_r^{\REV{n_r}}\bignorm{\partial_{x_r}^{\REV{n_r}} f}_{L_\infty(H)},
	\end{align}
	where
	\begin{align*}
		L_r^d
		= \frac{\norm{(\cdot-\gamma_r^{\REV{1}}) \cdot \ldots \cdot (\cdot-\gamma_r^{n_r})}_{L_\infty{[0,h_r]}}}{h_r^{\REV{n_r}} n_r!} \, \norm{\REV{\mathcal{L}}_{r+1}}_{L_\infty{[0,h_{r+1}]}} \cdot\ldots\cdot \norm{\REV{\mathcal{L}}_d}_{L_\infty{[0,h_d]}},
	\end{align*}
	and $\REV{\mathcal{L}}_r = \sum_{i=1}^{n_r} |\REV{\mathcal{L}}_r^i|$ denotes the Lebesgue function.
\end{theorem}

For $N$ Chebyshev-Gauss-Lobatto points
\begin{align*}
	\gamma^i = \cos\!\left(\frac{\pi \REV{(i-1)}}{N\REV{-1}}\right),\quad i = \REV{1},\ldots,N,
\end{align*}
in each dimension (that is, $n_r = N$ for $r = 1,\ldots,d$), \eqref{eq:interpolation_error} can be simplified due to the estimate
\REV{%
\begin{align*}
	\norm{\REV{\mathcal{L}}_{r}}_{L_\infty{[0,h_{r}]}}
	\leq \frac{2}{\pi} \log(N) + 1
\end{align*}
from \cite[Thm.~15.2]{Trefethen2019} for all $r = 1,\ldots,d$ since it is invariant under affine transformations. Furthermore it is well-known (see, for example \cite[Ex.~4.1]{Trefethen2019}), that
\begin{align*}
	\norm{(\cdot-\gamma_r^1) \cdot \ldots \cdot (\cdot-\gamma_r^N)}_{L_\infty{[0,h_r]}} \leq 2^{2-N} \left(\frac{h_r}{2}\right)^{\!N}
\end{align*}
where $\gamma_r^i$ denote the scaled and shifted versions of $\gamma^i$ onto $[0,h_r]$ for $i = 1,\ldots,N$. Combining both estimates gives
}
\begin{align*}
	L_r^d
	\leq \frac{1}{4^{\REV{N-1}} \REV{N}!} \REV{\prod_{q=r+1}^{d} \left( \frac{2}{\pi} \log(N) + 1 \right)}
	= \frac{1}{4^{N\REV{-1}} N!} \left( \frac{2}{\pi} \log(N) + 1 \right)^{d-r}\!\!.
\end{align*}
This yields the error bound
\begin{align*}
	\norm{f - I[f]}_{L_\infty(H)}
	\leq \frac{1}{4^{N\REV{-1}} N!} \sum_{r=1}^d \left( \frac{2}{\pi} \log(N) + 1 \right)^{d-r} h_r^{\REV{N}}\bignorm{\partial_{x_r}^{\REV{N}} f}_{L_\infty(H)}.
\end{align*}
Since the functions we consider are smooth on each element, the sum of $\norm{\cdot}_{\Tspace\REV{(\Omega_T)}}$-errors can be brought below any chosen $\tol_\mathrm{int}>0$ by choosing $N$ sufficiently large.

By combining the previous observations with the results from \cref{sec:nonlin,sec:lipschitz}, we obtain the following main result on the convergence of the adaptive scheme \REV{as an immediate consequence}.
\begin{theorem}\label{thm:porodiscrerr}
	With $\poroOpLipschitz$ from \cref{lem:poroContraction}, let $\errorRedFac<1-\poroOpLipschitz$ and let $\tol_\mathrm{proj}$, $\tol_\mathrm{int}$, $\tol_\press$ be chosen such that
	\begin{align*}
		\bignorm{\discrErr{k}}_{\Tspace\REV{(\Omega_T)}}
		\;\REV{\leq}\; \tol_\mathrm{proj} + T \, \tol_\mathrm{int} + \REV{C} \, \tol_\press \, (Q + T c_L)
		\leq \errorRedFac \, (\poroOpLipschitz + \errorRedFac)^k \, \bignorm{\poroGeneral - \poroGeneralDiscr^{(0)}}_{\Tspace\REV{(\Omega_T)}}\,,
	\end{align*}
	holds for all $k\leq N$ with \REV{$C$ coming from~\eqref{eq:errestimatenonlin}}. Then $\bignorm{\poroGeneral - \poroGeneralDiscr^{(k)}}_{\Tspace\REV{(\Omega_T)}} \leq (\poroOpLipschitz + \errorRedFac)^{k} \bignorm{\poroGeneral - \poroGeneralDiscr^{(0)}}_{\Tspace\REV{(\Omega_T)}}$ for $k \leq N+1$, where $\poroGeneral$ solves \eqref{eq:mildweak} and $\poroGeneralDiscr^{(k)}$ is defined as in~\eqref{eq:poroDiscr}.
\end{theorem}
\REV{
Hence by choosing $k$ large enough, it is possible to achieve an arbitrarily small $\Tspace$-error for the approximation $\poroGeneralDiscr$. This finally allows us to estimate the error for $\pressDiscr$.
\begin{corollary}\label{cor:pressdiscrerr}
	Under assumptions of \cref{prp:nonlinerr,thm:porodiscrerr}, let $\poroGeneralDiscr = \poroGeneralDiscr^{(k)}$ and $\paraOpDiscr[\poroGeneralDiscr]$ be the corresponding numerical solutions such that
	\begin{align*}
		\norm{\paraOp[\poroGeneralDiscr]-\paraOpDiscr[\poroGeneralDiscr]}_{\Tspace(\Omega_T)} \lesssim \tol_{\press} \,,\quad
		\epsilon_{\poroGeneral} = \norm{\poroGeneral-\poroGeneralDiscr}_{\Tspace(\Omega_T)} \lesssim \tol_{\poroGeneral} \,.
	\end{align*}
Then it holds that
	\begin{align*}
		\epsilon_{\press}
		= \norm{\paraOp[\poroGeneral] - \paraOpDiscr[\poroGeneralDiscr]}_{\Tspace(\Omega_T)}
		&\lesssim T^{1/2} \tol_{\poroGeneral} + \tol_{\press} \,.
	\end{align*}
\end{corollary}
\begin{proof}
	By applying ~\eqref{eq:estpress1} and~\eqref{eq:estpress2} we get
	\begin{align*}
		\norm{\paraOp[\poroGeneral] - \paraOpDiscr[\poroGeneralDiscr]}_{\Tspace(\Omega_T)}
		&\leq \norm{\paraOp[\poroGeneral] - \paraOp[\poroGeneralDiscr]}_{\Tspace(\Omega_T)} + \norm{\paraOp[\poroGeneralDiscr] - \paraOpDiscr[\poroGeneralDiscr]}_{\Tspace(\Omega_T)}\\
		&\lesssim T^{1/2} \norm{\poroGeneral - \poroGeneralDiscr}_{\Tspace(\Omega_T)} + \norm{\paraOp[\poroGeneralDiscr] - \paraOpDiscr[\poroGeneralDiscr]}_{\Tspace(\Omega_T)}\\
		&\lesssim T^{1/2} \tol_{\poroGeneral} + \tol_{\press} \,,
	\end{align*}
	which is the desired result.
\end{proof}
Note that this result is different from the one in \cref{thm:pressdiscrerr}, since there we only considered the error $\norm{\paraOp[\overline{\poroGeneral}]-\paraOpDiscr[\overline{\poroGeneral}]}_{\Tspace(\Omega_T)}$ for a fixed $\overline{\poroGeneral}$.
}

\subsection{Time slices}\label{sec:timeslices}%
As mentioned in \cref{sec:temporal}, we will in general need to split the domain into time slices in order to ensure convergence of the nonlinear iterations.
This aligns with the theory in \cite{Bachmayr2023}, where existence results are local in time. We thus obtain a natural limitation on the largest time steps that can be produced by the adaptive scheme. However, the size of the slices does not depend on the discretizations.

Let us now consider the propagation of solution errors in such a scheme. 
To simplify notation, in the following discussion we let $[0,T]$ stand for a time slice of admissible size.
In view of \eqref{eq:imbedding}, we have a well-defined exact solution $( \varphi,  u) \in (C([0,T];L_2(\Omega)))^2$ of \eqref{eq:mildweak} on $[0,T]$ for initial data $( \varphi_0,  u_0) \in (L_2(\Omega) )^2$ at $t=0$. 

By Lipschitz continuity of $( \varphi,  u)$ with respect to $( \varphi_0,  u_0)$ and the embedding~\eqref{eq:imbedding}, $(\gamma_{T} \tilde\varphi, \gamma_{T}\tilde u) \in (L_2(\Omega))^2$ is Lipschitz continuous with respect to $( \varphi_0,  u_0)$ with a constant depending on the problem data and on $T$, which here is the length of the time slice. In particular, $(\gamma_{T} \varphi, \gamma_{T} u)$ provide initial data for the following time slice.

Let us now consider the numerical approximations $\poroGeneralDiscr$ of $\varphi$ and $ \paraOpDiscr[\poroGeneralDiscr]$ of $\press$, respectively.
Note first that by \cref{thm:pressdiscrerr,thm:porodiscrerr} together with~\eqref{eq:estpress1} and~\eqref{eq:estpress2} the $L_2$-errors of $\gamma_{T} \poroGeneralDiscr$ and $\gamma_{T} \paraOpDiscr[\poroGeneralDiscr]$ in the initial slice are controlled \REV{since by \cref{thm:porodiscrerr,cor:pressdiscrerr} we know that
\begin{align*}
	\norm{\gamma_{T} \poroGeneralDiscr - \gamma_{T} \poroGeneral}_{L_2(\Omega)}
	&\leq \norm{\poroGeneralDiscr-\poroGeneral}_{\Tspace(\Omega_T)}
	\lesssim \tol_{\poroGeneral} \,,\\
	\norm{\gamma_{T} \pressDiscr - \gamma_{T} \press}_{L_2(\Omega)}
	&\leq \norm{\pressDiscr-\press}_{\Tspace(\Omega_T)}
	\lesssim T^{1/2} \tol_{\poroGeneral} + \tol_{\press} \,,
\end{align*}
where the hidden constants are independent of $\press,\pressDiscr,\poroGeneral,\poroGeneralDiscr$.
}

On the following time slices we can estimate the errors by the sum of the nonlinear error, the amplified initial error (which equals the terminal error of the previous slice) and the discretization error due to a modified version of~\eqref{eq:errinductive} including initial errors.
\REV{
Assuming that the nonlinear errors from the previous slice are given by
\begin{align*}
	\epsilon_{\poroGeneral}^{\mathrm{old}}
	= \norm{\poroGeneralDiscr^{\mathrm{old}}-\poroGeneral^{\mathrm{old}}}_{\Tspace(\Omega_T)} \,,\quad
	\epsilon_{\press}^{\mathrm{old}}
	= \norm{\pressDiscr^{\mathrm{old}}-\press^{\mathrm{old}}}_{\Tspace(\Omega_T)} \,,
\end{align*}
one can proceed as in \cref{thm:pressdiscrerr,thm:porodiscrerr,cor:pressdiscrerr}. The difference, however, is that we need to incorporate perturbations in the initial errors as well. We obtain
\begin{align*}
	\norm{\paraOp[\poroGeneralDiscr]-\paraOpDiscr[\poroGeneralDiscr]}_{\Tspace(\Omega_T)}
	\leq \norm{\paraOp[\poroGeneralDiscr] - \breve{\paraOp}[\poroGeneralDiscr]}_{\Tspace(\Omega_T)} + \norm{\breve{\paraOp}[\poroGeneralDiscr] - \paraOpDiscr[\poroGeneralDiscr]}_{\Tspace(\Omega_T)}
	\lesssim \epsilon_{\press}^{\mathrm{old}} + \tol_{\press} \,,
\end{align*}
where we proceed as in \cref{thm:pressdiscrerr} for the second summand and
$\breve{\paraOp}$ is defined to be the exact solution operator solving~\eqref{eq:weak} but with the perturbed initial function $\press_{0,\discr}$. The first estimate follows directly from parabolic regularity theory (see, for example \cite[Chap.~7]{Evans2010}). Continuing as with \cref{thm:porodiscrerr} (using \cref{lem:poroContraction} as well) we define $\breve{\poroGeneral}$ to be the exact solution of~\eqref{eq:mild} but with perturbed initial data given by $\poroGeneral_{0,\discr} = \poroGeneralDiscr^{\mathrm{old}}(T,\cdot)$ and $\press_{0,\discr} = \pressDiscr^{\mathrm{old}}(T,\cdot)$. Then by~\eqref{eq:estpress1} and~\eqref{eq:estpress2},
\begin{align*}
	\norm{\poroGeneral-\breve{\poroGeneral}}_{\Tspace(\Omega_T)}
	&= \bigg\| \poroGeneral_{0} - \poroGeneral_{0,\discr} + Q (\press_{0}-\press_{0,\discr}) - Q (\paraOp[\poroGeneral]-\breve{\paraOp}[\breve{\poroGeneral}])\\
	&\phantom{=\bigg\|} - \int_{0}^{t} \beta(\poroGeneral) \kappa(\paraOp[\poroGeneral]) - \beta(\breve{\poroGeneral}) \kappa(\breve{\paraOp}[\breve{\poroGeneral}]) \D s \bigg\|_{\Tspace(\Omega_T)}\\
	&\lesssim \epsilon_{\poroGeneral}^{\mathrm{old}} + \epsilon_{\press}^{\mathrm{old}} + T^{1/2} \norm{\poroGeneral-\breve{\poroGeneral}}_{\Tspace(\Omega_T)} \,,
\end{align*}
which yields
\begin{align*}
	\norm{\poroGeneral-\breve{\poroGeneral}}_{\Tspace(\Omega_T)}
	\leq \frac{C}{1-T^{1/2}C} (\epsilon_{\poroGeneral}^{\mathrm{old}} + \epsilon_{\press}^{\mathrm{old}}) 
\end{align*}
for sufficiently small $T$. Combing this with \cref{thm:porodiscrerr} leads to
\begin{equation}\label{eq:propagationphi}
	\norm{\poroGeneral-\poroGeneralDiscr}_{\Tspace(\Omega_T)}
	\leq \norm{\poroGeneral-\breve{\poroGeneral}}_{\Tspace(\Omega_T)} + \norm{\breve{\poroGeneral}-\poroGeneralDiscr}_{\Tspace(\Omega_T)}
	\lesssim \epsilon_{\poroGeneral}^{\mathrm{old}} + \epsilon_{\press}^{\mathrm{old}} + \tol_{\poroGeneral} \,.
\end{equation}
Applying the same argument as in \cref{cor:pressdiscrerr} finally yields
\begin{equation}\label{eq:propagationu}
	\norm{\press-\pressDiscr}_{\Tspace(\Omega_T)}
	\leq \norm{\paraOp[\poroGeneral] - \paraOp[\poroGeneralDiscr]}_{\Tspace(\Omega_T)} + \norm{\paraOp[\poroGeneralDiscr]-\paraOpDiscr[\poroGeneralDiscr]}_{\Tspace(\Omega_T)}
	\lesssim T^{1/2} \tol_{\poroGeneral} + \epsilon_{\press}^{\mathrm{old}} + \tol_{\press} \,,
\end{equation}
which shows that we get an amplification due to the perturbed initial data.
}

\subsection{Adaptive choice of tolerances}\label{sec:adapt_tol}
Now we want to use the above framework to choose the tolerances of \cref{alg:coupled} adaptively. We assume that the interpolation order is sufficiently high, so that $\tol_\mathrm{int}$ can be neglected. Furthermore, the underlying Lipschitz constants need to satisfy $\poroOpLipschitz, \pressOpLipschitz < 1$, which amounts to a restriction on the sizes of time slices.

We start by solving the nonlinear equation for $\press$ with an initial guess of the respective Lipschitz constant and optionally update it during the iteration.
Then we use the same idea to solve for $\poroGeneral$ which yields an improved version of \cref{alg:coupled}.
\begin{algorithm}[h!]
	\caption{\textsc{Fully Adaptive Method}}\label{alg:adapt_tol}
	\begin{algorithmic}
		\REQUIRE $\tol_\poroGeneral$, $\tol_\mathrm{int}$, $\poroGeneral_0$, $\press_0$, $\theta \in (0,1)$, $\errorRedFac_\poroGeneral \in (0,1)$, $\errorRedFac_\press \in (0,1)$,\\
		initial guesses $\poroOpLipschitz \in (0,1)$, $\pressOpLipschitz \in (0,1)$
		\ENSURE $\poroGeneralDiscr$, $\pressDiscr$
		\STATE initialize $\poroGeneralDiscr^{(0)}\REV{(t,x) = \poroGeneral_0(x), \res_\poroGeneral = \tol_\poroGeneral+1, k = 0}$
		\WHILE{$\res_\poroGeneral > \tol_\poroGeneral$}
		\STATE \REV{set $\tol_\press = \theta \frac{1}{Q + T c_L} \errorRedFac_\poroGeneral (1-\poroOpLipschitz) \res_\poroGeneral$}
		\STATE \REV{set $\tol_\mathrm{proj} = (1-\theta) \errorRedFac_\poroGeneral (1-\poroOpLipschitz) \res_\poroGeneral$}
		\STATE initialize $\pressDiscr^{(0)}\REV{(t,x) = \press_0(x), \res_\press = \tol_\press+1, \ell = 0}$
		\WHILE{$\res_\press > \tol_\press$}
		\STATE set $\tol_\mathrm{lsq} = \errorRedFac_\press (1-\pressOpLipschitz) \res_\press$
		\STATE solve $\REV{(\pressDiscr^{(\ell+1)},\pressFluxDiscr^{(\ell+1)})} = \paraOpDiscr[\poroGeneralDiscr^{(k)},\pressDiscr^{(\ell)}]$ up to $\tol_\mathrm{lsq}$ \REV{by means of~\eqref{eq:lsqlin}}
		\STATE compute $\res_\press = \norm{G[\pressDiscr^{(\ell+1)}]\REV{(\pressDiscr^{(\ell+1)},\pressFluxDiscr^{(\ell+1)})} - R}_\pressRanSp$
		\STATE \textit{Optional:} update Lipschitz constant $\pressOpLipschitz$
		\STATE \REV{$\ell = \ell+1$}
		\ENDWHILE
		\STATE calculate $\poroGeneralDiscr^{(k+1)}$ by \eqref{eq:poroDiscr} up to $\tol_\mathrm{proj}$, $\tol_\mathrm{int}$
		\STATE \REV{compute $\res_\poroGeneral = \frac{\poroOpLipschitz}{1-\poroOpLipschitz} \bignorm{\poroGeneralDiscr^{(k+1)}-\poroGeneralDiscr^{(k)}}_{\Tspace\REV{(\Omega_T)}}$}
		\STATE \textit{Optional:} update Lipschitz constant $\poroOpLipschitz$
		\STATE \REV{$k = k+1$}
		\ENDWHILE
	\end{algorithmic}
\end{algorithm}
We adapt $\tol_\poroGeneral$ and $\tol_\mathrm{proj}$ here and use the nonlinear residual as error indicator for $\press$ as before. But in contrast to \cref{def:paraOp,alg:coupled}, we can \REV{generally} obtain a better \REV{approximation of the error} of $\poroGeneral$ \REV{by the indicator $\res_{\poroGeneral}$} due to the Lipschitz constant $\poroOpLipschitz$. \REV{Although it might} still not be very accurate depending on the value of $\poroOpLipschitz$, it will allow us to observe the order of convergence in \cref{sec:error}.
\begin{remark}\label{rem:LipschitzUpdate}
	Various heuristics are possible for updating the estimates of Lipschitz constants.
	In our tests, we update the Lipschitz constants as
	\REV{
	\begin{align*}
		\pressOpLipschitz = \min \bigg\{ \pressOpLipschitz^{\mathrm{init}}, \frac{\res_{\press}^\mathrm{new}}{\res_{\press}^\mathrm{old}} \bigg\},\quad
		\poroOpLipschitz = \min \bigg\{ \poroOpLipschitz^{\mathrm{init}}, \frac{\res_{\poroGeneral}^\mathrm{new}}{\res_{\poroGeneral}^\mathrm{old}} \bigg\},
	\end{align*}
	to avoid inadmissible constants greater or equal to one.}
	This can still occur, but only if the initial guess for $\pressOpLipschitz$ is too small and hence the discretization error is so large that the discrete fixed-point map is no longer a contraction. In practice, we start updating the Lipschitz constants only after several iteration steps to obtain stable estimates.
\end{remark}
Convergence of \cref{alg:adapt_tol}, provided that $\errorRedFac_\poroGeneral$ and $\errorRedFac_\press$ are sufficiently small, follows directly from \cref{thm:pressdiscrerr,thm:porodiscrerr}.

\section{Numerical Experiments} 
In this section, we present numerical results obtained by the space-time adaptive method described in \cref{alg:coupled,alg:adapt_tol}. We first show results for different test cases and then turn to convergence rates of the fully adaptive method as described in \cref{alg:adapt_tol}. \REV{Note that all tolerances we chose to run \cref{alg:coupled,alg:adapt_tol} can be found in \cref{tab:alg_params}. Throughout this section we use polynomial degrees $(\REV{q,p})=(2,3)$ for the finite element space as in \eqref{eq:fespace}.}

\subsection{Applications}\label{sec:applications}
\Cref{alg:coupled,alg:adapt_tol} do not require $\poroGeneral$ or $\poro$ to be continuous and are thus in particular applicable to problems with discontinuities in $\poroGeneral_0$ or $\poro_0$.
We first consider the test problem
\begin{subequations}\label{eq:fullNonLin}
	\begin{align}
		\partial_t \poro &=-(1-\poro) \bl \frac{\poro}{\sigma(\press)}\press +Q\partial_t \press \br ,\\
		\partial_t \press &= \nabla \cdot \poro^3 (\nabla \press + (1-\poro))-\frac{\poro}{\sigma(\press)}\press ,
	\end{align}
\end{subequations}
with $\Omega = (0,1)$ and $T = 1$ where
$
	\sigma(\press)
	= 1 - \REV{\tfrac{12}{25}} \bl 1 + \tanh \bl - 25 \press \br \br
$.
Here we also make use of the reformulation~\eqref{eq:logtransformed} in order to deal with the factor $(1-\poro)$.
\begin{figure}[h!]
	\centering
	\vspace*{-0.3cm}
	\capstart
	\FPmul\result{0.95}{\convfac}
	\includegraphics[width=\result\textwidth]{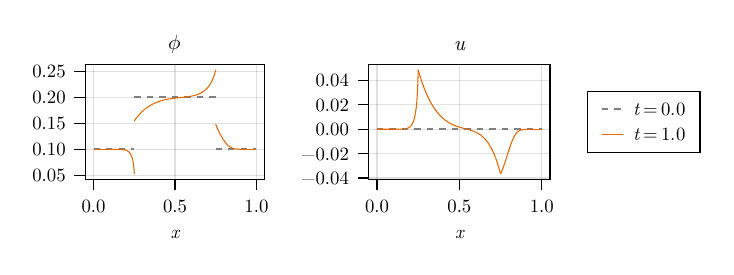}
	\vspace*{-0.6cm}
	\caption{Numerical approximation of $\poro$ and $\press$ from~\eqref{eq:fullNonLin}.}\label{fig:fullNonLin}
\end{figure}%
In \cref{fig:fullNonLin} one can see the numerical solution of \cref{alg:coupled} for the initial and terminal time and the corresponding space-time grids are shown in \cref{fig:fullNonLinGrid}.
\begin{figure}[h!]
	\centering
	\vspace*{-0.3cm}
	\capstart
	\FPmul\result{0.9}{\convfac}
	\includegraphics[width=\result\textwidth]{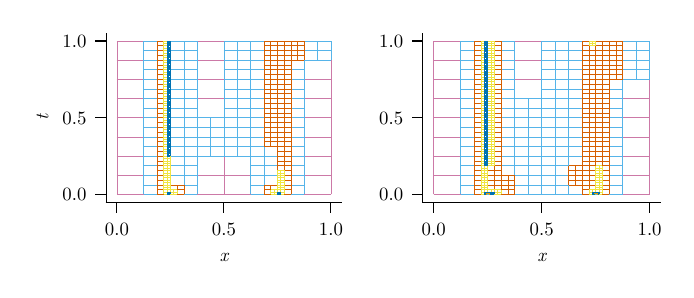}
	\vspace*{-0.6cm}
	\caption{Space-time grids for $\poro$ (left) and $\press$ (right) from~\eqref{eq:fullNonLin}\REV{, the colors indicate the refinement level of each element.}}\label{fig:fullNonLinGrid}
\end{figure}%
This highlights the localized behavior of solutions and hence shows the advantage of space-time adaptivity in this context. It also shows the formation of steep gradients in $\poro$ near discontinuities in the initial data.

Next we apply \cref{alg:coupled} to a more realistic problem from geophysics. \REV{Therefore we consider \cite[eq.~7]{Vasilyev1998} with decompaction weakening as in~\cite{Raess2018,Raess2019} and $Q=3\cdot10^9\,\mathrm{Pa}$ as in~\cite{Connolly1998}. To nondimensionalize it we use the independent scales
\begin{align}\label{eq:indep_scales}
	\scale{x}
	= 10^4\,\mathrm{m} \,,\quad
	\scale{\densityDiff} \scale{g}
	= 5\cdot10^3\,\mathrm{kg}\cdot\mathrm{m}^{-2}\cdot\mathrm{s}^{-2} \,,\quad
	\scale{\bulkVisc}
	= 10^{19}\,\mathrm{Pa}\cdot\mathrm{s} \,,
\end{align}
which yield
\begin{align}
	\begin{aligned}\label{eq:scales}
		\scale{\press}
		&= \scale{\densityDiff} \scale{\gravity} \scale{x}
		= 5\cdot10^7\,\mathrm{Pa} \,,
		&&\scale{t}
		= \frac{\scale{\bulkVisc}}{\scale{\press}}
		= 2\cdot10^{11}\,\mathrm{s} \,,\\
		\frac{\scale{\perm}}{\scale{\fluidVisc}}
		&= \frac{\scale{x}^2}{\scale{\bulkVisc}}
		= 10^{-11}\,\mathrm{m}^2\cdot\mathrm{Pa}^{-1}\cdot\mathrm{s}^{-1} \,,
		&&\scale{Q}
		= \frac{1}{\scale{\press}} = 2\cdot10^{-8} \,\mathrm{Pa}^{-1} \,.
	\end{aligned}
\end{align}
This results in the nondimensional number $Q=\frac{1}{60}$, again denoted as $Q$ for convenience, as well as $a(\phi) = (10\,\phi)^3$ due to the background porosity (see, for example~\cite{Connolly1998,Raess2019}).}

For the first test, we consider a discontinuous $\poro_0$ and $\sigma \equiv 1$ (corresponding to no decompaction weakening).
The equations in nondimensional form read
\begin{subequations}\label{eq:geoTest}
	\begin{align}
		\partial_t \poro &=-(1-\poro) \bl \poro \, \press +\frac{1}{60} \partial_t \press \br ,\\
		\partial_t \press &= 60 \, \REV{\bl \nabla \cdot (10 \, \poro)^3 (\nabla \press + (1-\poro))-\poro \, \press \br} ,
	\end{align}
\end{subequations}
for $\Omega = (0,3)$ and $T = 15.77\REV{88}$, which represent a length of $30\,\mathrm{km}$ and time of $0.1 \,\mathrm{Myr}$ after rescaling \REV{with $\scale{x}$ and $\scale{t}$ from~\eqref{eq:indep_scales} and~\eqref{eq:scales}}.

As discussed in \cref{sec:temporal,sec:timeslices}, we can split the space-time cylinder into time slices whose size solely depends on the continuous problem (via the Lipschitz constant of $\Theta$), but not on the discretization.
\begin{figure}[h!]
	\centering
	\vspace*{-0.3cm}
	\capstart
	\includegraphics[width=\convfac\textwidth]{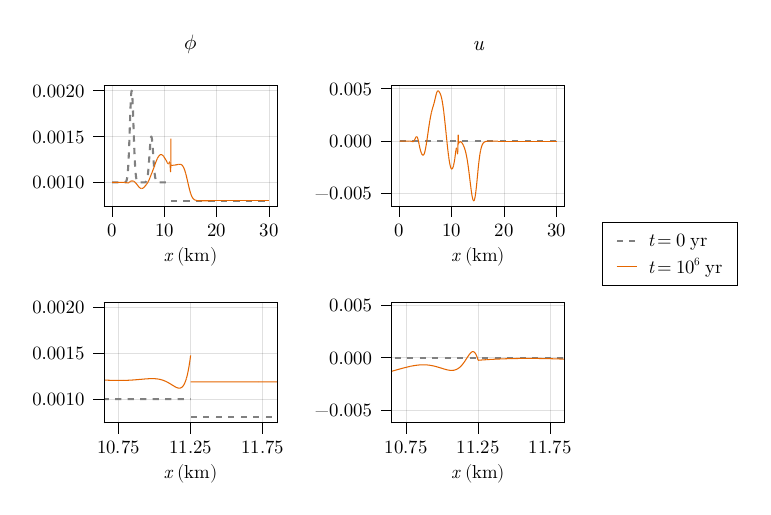}
	\vspace*{-0.6cm}
	\caption{Numerical approximation of $\poro$ and $\press$ from~\eqref{eq:geoTest} with zoom-in at the discontinuity (bottom).}\label{fig:geoTest2b}
\end{figure}%
The corresponding grids for the different slices are concatenated to obtain two separate global space-time grids for $\poro$ and $\press$, which are shown in \cref{fig:geoTest2b_grid}; note the localized refinements both in space and in time.
Hence we still obtain a space-time adaptive method with the advantage of localized time-steps.

\REV{We solve~\eqref{eq:geoTest} for a discontinuous $\poro_0$ modelling two porosity anomalies and a sharp interface. Its precise definition can be found in \cref{tab:init_setup}. This}
yields \REV{the} results depicted in \cref{fig:geoTest2b} with the associated grids shown in \cref{fig:geoTest2b_grid}.
Here we plot the numerical solution at the start and after 10 time slices.
One can clearly see the similarities with the test problems considered in \cref{fig:comparison,fig:fullNonLin}, for example that discontinuities lead to the formation of steep gradients. This is particularly visible in the bottom of \cref{fig:geoTest2b}, where we show the solution near the discontinuity.
This shows the advantage of our adaptive method in resolving solution features on different scales, which is reflected in the corresponding grids shown in \cref{fig:geoTest2b_grid}.
\begin{figure}[h!]
	\centering
	\vspace*{-0.3cm}
	\capstart
	\FPmul\result{0.9}{\convfac}
	\includegraphics[width=\result\textwidth]{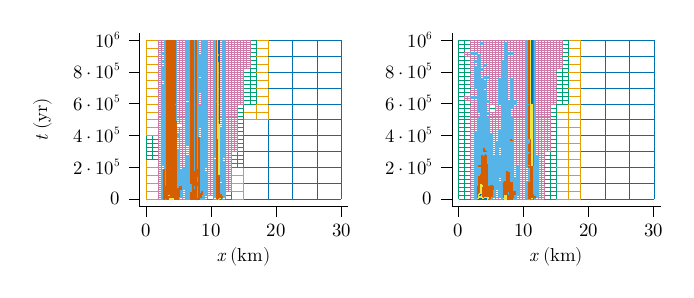}
	\vspace*{-0.6cm}
	\caption{Space-time grids for $\poro$ (left) and $\press$ (right) corresponding to \cref{fig:geoTest2b}.}\label{fig:geoTest2b_grid}
\end{figure}%
The gradients near discontinuities that become increasingly pronounced with time lead to further refinement of the grid near the location of the discontinuity.

Increasing the complexity we stay in a geophysical regime but consider the full nonlinear problem with decompaction weakening and a discontinuity, which reads
\begin{subequations}\label{eq:channel}
	\begin{align}
		\partial_t \poro &=-(1-\poro) \bigg( \frac{\poro^2}{\sigma(\press)}\press + \frac{1}{60}\partial_t \press \bigg) ,\\
		\partial_t \press &= 60 \,\REV{\bl \nabla \cdot (10\poro)^3 \bigg(\nabla \press + (1-\poro) \begin{pmatrix}0\\1\end{pmatrix}\bigg)-\frac{\poro^2}{\sigma(\press)}\press \br } ,
	\end{align}
\end{subequations}
similar to~\eqref{eq:geoTest} with
$
	\sigma(\press)
	= 1 - \tfrac{999}{2000} \bl 1 + \tanh \bl - 500 \press \br \br
$,
for $d=2$, $\Omega = (0,1)\times(0,2)$ and $T = 15.7788$, which represent a length of $10\,\mathrm{km}\times20\,\mathrm{km}$ and time of $0.1 \,\mathrm{Myr}$ after rescaling \REV{with $\scale{x}$ and $\scale{t}$ from~\eqref{eq:indep_scales} and~\eqref{eq:scales}}. In this case, the solution exhibits channel formation as shown in \cref{fig:channel}.
\begin{figure}[h!]
	\centering
	\vspace*{-0.0cm}
	\capstart
	\includegraphics[width=\textwidth]{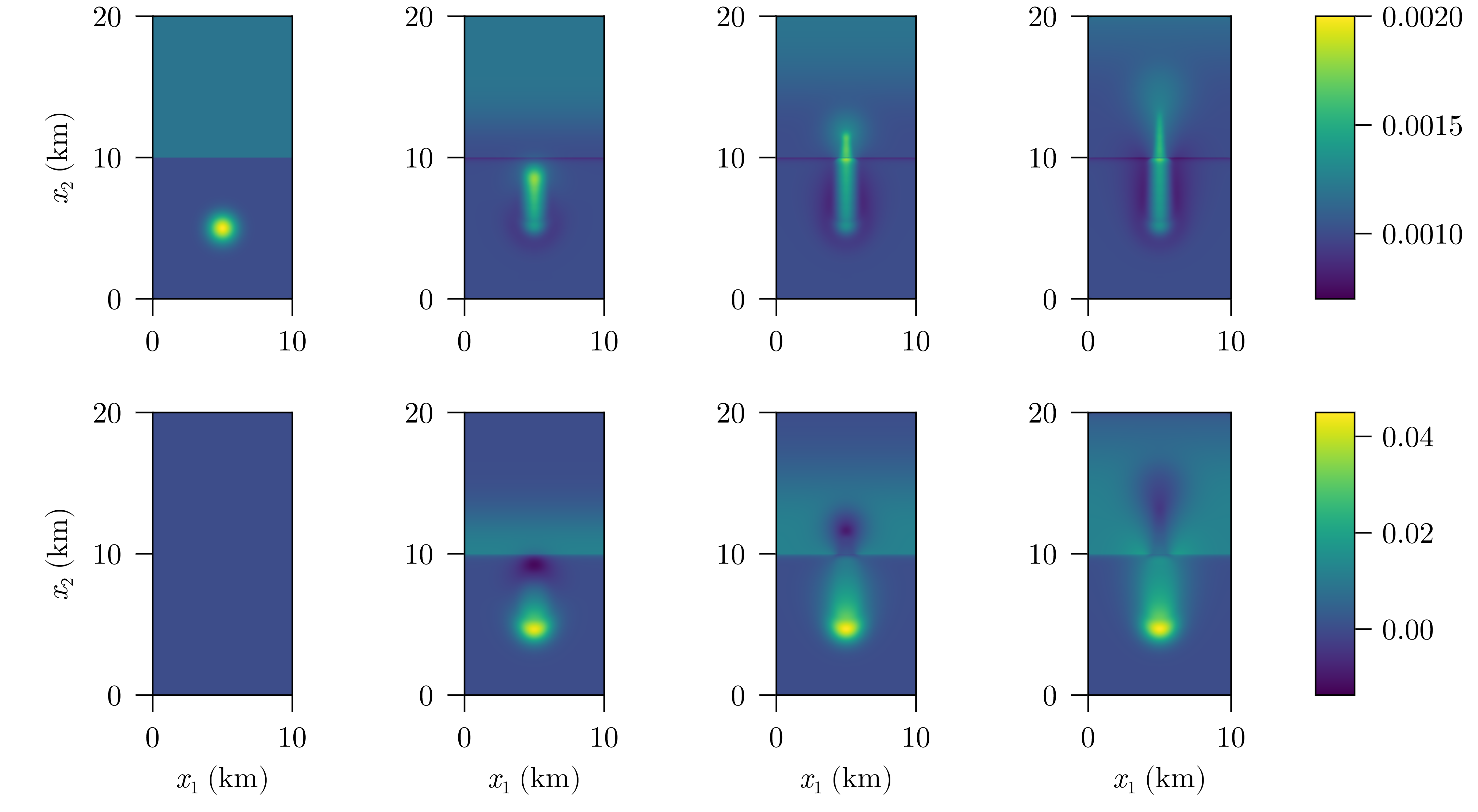}
	\vspace*{-0.5cm}
	\caption{Numerical approximation of $\poro$ (top) and $\press$ (bottom) for $t=0,0.5,1,1.5 \,\mathrm{Myr}$ (from left to right).}\label{fig:channel}
\end{figure}%
\Cref{fig:channel_grid} shows the corresponding three-dimensional space-time grids. Here we use 15 time slices and the linearization described in \cref{rem:taylor}.
\begin{figure}[h!]
	\centering
	\capstart
	\FPmul\result{1.}{\convfac}
	\includegraphics[width=\result\textwidth]{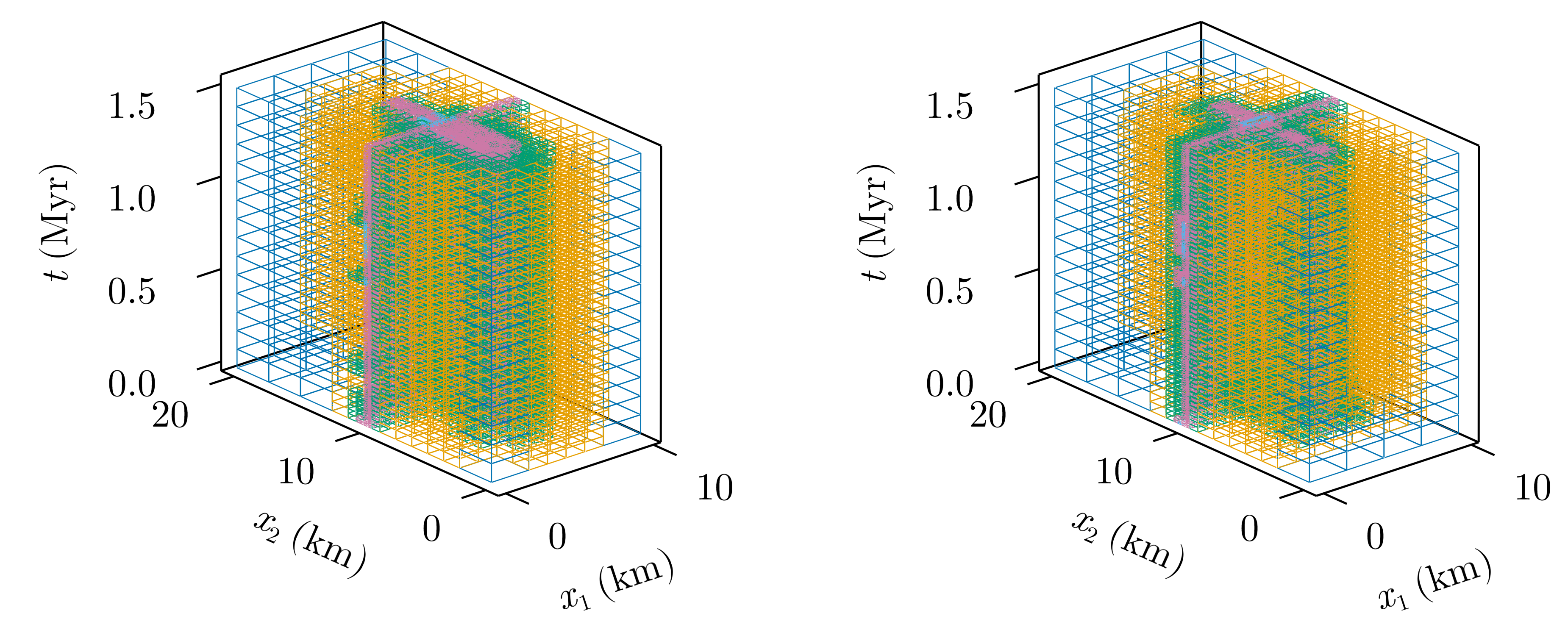}
	\vspace*{-0.0cm}
	\caption{Space-time grids for $\poro$ (left) and $\press$ (right) corresponding to \cref{fig:channel}.}\label{fig:channel_grid}
\end{figure}%
\REV{Note that we consider mixed Dirichlet-Neumann boundary conditions for this numerical test. More precisely, the respective boundary parts $\Gamma_N, \Gamma_D \subset \partial \Omega$, $\Gamma_N \cup \Gamma_D = \partial \Omega$ in these tests are equal to
\begin{align*}
	\Gamma_D = \{(x_1,x_2) \in \partial\Omega \colon x_2 = 0 \text{ or } x_2 = 2\},\quad
	\Gamma_N = \{(x_1,x_2) \in \partial\Omega \colon x_1 = 0 \text{ or } x_1 = 1\}.
\end{align*}
which implies that} we consider homogeneous Neumann boundary conditions on the left and right part of the boundary, as well as homogeneous Dirichlet conditions on the upper and lower part of the boundary.
\begin{remark}
	Homogeneous Neumann boundary conditions in the least squares formulation \eqref{eq:leastsquares} of the parabolic problems \eqref{eq:weak} with source terms that are only in $L_2(0,T;H^{-1}(\Omega))$ rather than $L_2(\Omega_T)$ generally pose some difficulties, since they correspond to  \emph{inhomogeneous} Dirichlet boundary conditions \REV{(see, for example \cite[Prop.~2.5]{Gantner2021})} on the normal flux $\pressFlux \cdot n$ on $\Gamma_\mathrm{N}\times[0,T]$, where $\Gamma_\mathrm{N} \subset \partial\Omega$ denotes the Neumann boundary and $n$ is the outward normal vector on $\Gamma_\mathrm{N}$. In \cite[Sec.~2.2]{Gantner2021} some more involved techniques are developed for this general case, for example by incorporating an appropriate boundary term into the error functional.
	
	In the particular case~\eqref{eq:channel}, however, we can exploit that the outward normals on $\Gamma_\mathrm{N}$, composed of the left and right sides of $\Omega$, are given by $n = (\pm1,0)^\top$ and thus $\zeta \cdot n = 0$. This implies that we can enforce the Neumann boundary condition on $\press$ by enforcing a \emph{homogeneous} Dirichlet condition on $\pressFlux \cdot n = -\Qalpha \nablax \press \cdot n$.
	
	\REV{To apply these boundary conditions as we did for homogeneous Dirichlet boundary in \cref{sec:parabolic} we replace all spaces $H_0^1(\Omega)$ there by $H_D^1(\Omega) = \{u \in H^1(\Omega) \colon u|_{\Gamma_D} = 0\}$ where $\Gamma_D$ denotes the Dirichlet part of the boundary. Note that \cref{thm:isomorphism} still holds for the adapted spaces due to \cite[Thm.~2.3]{Gantner2021}.}
\end{remark}

\subsection{Convergence rates}\label{sec:error}
Using the fully adaptive methods described in \cref{alg:adapt_tol}, we now turn to the convergence rates achieved for some of the test problems considered so far.

\REV{But before we start with an abstract setting of a domain $\mathcal{D}$, a triangulation $\mathcal{T}$ and a conforming finite element space
\begin{align*}
	\mathcal{V}_{\discr}^{M}(\mathcal{T}) = \left\{ f_{\delta} \in \mathcal{V} \,\colon f_{\delta}|_{\textbf{I}} \in \polySpMult_{M}(\textbf{I}) \text{ for each } \textbf{I} \in \mathcal{T} \right\} \,,
\end{align*}
where $\mathcal{V} \in \{ L_{2}(\mathcal{D}), H^{1}(\mathcal{D}) \}$.
Furthermore we have the inverse inequality (see, for example \cite[Thm.~3.2.6]{Ciarlet2002}, \cite[Cor.~3.4.2]{Cohen2003} or \cite[Lem.~12.1]{Ern2021})
\begin{align}\label{eq:inverse}
	\norm{f_{\delta}}_{H^{M+1}(\mathcal{D})}
	\lesssim h_{\textbf{I}}^{l-M-1} \norm{f_{\delta}}_{H^{l}(\mathcal{D})} \,,
\end{align}
if $f_{\delta} \in \mathcal{V}_{\delta}^{M}(\mathcal{T}) \cap H^{M+1}(\mathcal{D})$ and $M+1 \geq l$ where $h_\textbf{I}$ denotes the diameter of $\textbf{I}$. Following the lines of \cite[Rem.~6.10]{Braess2013} we conclude that the convergence rate of $f_{\delta}\in\mathcal{V}_{\delta}^{M}(\mathcal{T})$ is bounded from above by $M+1-l$.}

\REV{Note that due to the embeddings
\begin{align*}
	H^1(0,T; H^1_0(\Omega)) \times L_2(0,T;H_{\Div_x}(\Omega)) \subset \pressSolSp \subset L_2(0,T; H^1_0(\Omega)) \times L_2(\Omega_T;\R^d)
\end{align*}
we can generally only expect the optimal rate in $H^1(0,T; H^1_0(\Omega)) \times L_2(0,T;H_{\Div_x}(\Omega))$, which is bounded from above by the optimal rate in $H^1(0,T; H^1(\Omega)) = H^1(\Omega_T)$.
}

\REV{Hence by~\eqref{eq:inverse} this convergence rate for the polynomial degrees $(\REV{q,p})=(2,3)$ in $\pressSolSpDiscr(\mathcal{T}_{\press})$ is bounded from above by $3$ (choosing $l=1$ in~\eqref{eq:inverse}), which translates to a rate of $\frac{3}{d+1} = \frac{3}{2}$ with respect to the number of degrees of freedom for the one-dimensional test cases. Moreover, by choosing $l=0$ in~\eqref{eq:inverse}, we see that the convergence rate for $\poroGeneralDiscr$ with polynomial degree $3$ in space and time is bounded from above by $4$, which translates to a rate $\frac{4}{d+1} = 2$ with respect to the total number of degrees of freedom.}

\REV{Next we assume that we have given the iterates $(\poroGeneralDiscr^{(k)},\pressDiscr^{(\ell_k)})$, $k = 0,\ldots,n$, $\ell_k = 0,\ldots,n_k$ from \cref{alg:adapt_tol}. We want to}
plot the \REV{relative} nonlinear error indicators
\REV{%
	\begin{align*}
		\mathrm{rel}_{\poroGeneral}^{(k)}
		= \frac{\norm{\poroGeneralDiscr^{(k)}-\poroGeneralDiscr^{(k-1)}}_{\Tspace(\Omega_T)}}{\norm{\poroGeneralDiscr^{(k)}}_{\Tspace(\Omega_T)}} \,,\qquad
		\mathrm{rel}_{\press}^{(k)}
		= \frac{\norm{\pressLinOp[\pressDiscr^{(n_k)}](\pressDiscr^{(n_k)},\pressFluxDiscr^{(n_k)})-\pressLinRhs}_\pressRanSp}{\norm{(\pressDiscr^{(n_k)},\pressFluxDiscr^{(n_k)})}_{\pressSolSp}} \,,
	\end{align*}
	for $k = 1,\ldots,n$ where we calculate the numerator of $\mathrm{rel}_{\press}^{(k)}$ as in~\eqref{eq:errestimatenonlin} using $\overline{\poroGeneral} = \poroGeneralDiscr^{(n)}$ since it is closest to the nonlinear solution~$\poroGeneral$ of~\eqref{eq:mild}.
}

\begin{figure}[h!]
	\centering
	\vspace*{-0.3cm}
	\capstart
	\includegraphics[width=\textwidth]{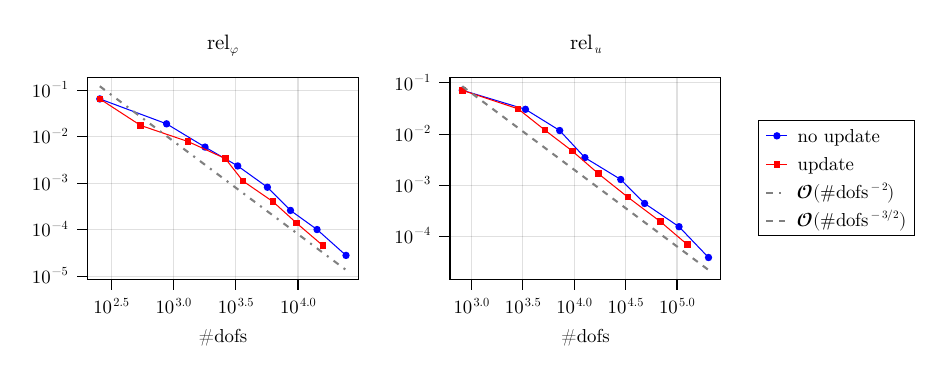}
	\vspace*{-1cm}
	\caption{\REV{Relative error indicators $\mathrm{rel}_{\poroGeneral}$, $\mathrm{rel}_{\press}$ for $\poroOpLipschitz = \frac35$ and $\pressOpLipschitz = \frac35$ corresponding to the solution of~\eqref{eq:fullNonLin} shown in \cref{fig:fullNonLin} with and without Lipschitz constant update.}}\label{fig:fullNonLinErr}
\end{figure}%
\REV{In \cref{fig:fullNonLinErr} one can see $\mathrm{rel}_{\poroGeneral}^{(k)}$ and for $\mathrm{rel}_{\press}^{(k)}$ for the nonlinear test case shown in \cref{fig:fullNonLin}.}
%
%
Here we observe that even in the presence of discontinuities, we obtain optimal convergence rates for the error of $\REV{\poroGeneralDiscr}$ in $\Tspace(\Omega_T)$ and the error of $\REV{\pressDiscr}$ in $U$ \REV{as it was mentioned above}. Furthermore, we observe that the error reduction improves \REV{if one} adaptively chang\REV{es the Lipschitz constant.}
%

\begin{figure}[h!]
	\centering
	\vspace*{-0.3cm}
	\capstart
	\includegraphics[width=\textwidth]{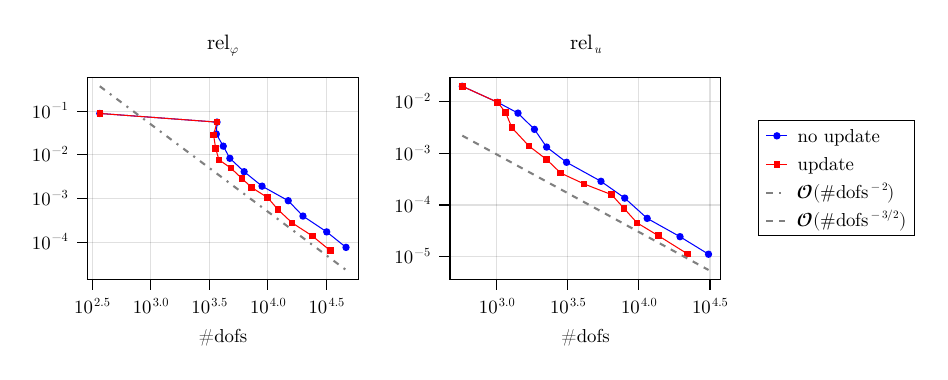}
	\vspace*{-1cm}
	\caption{\REV{Relative error indicators $\mathrm{rel}_{\poroGeneral}$, $\mathrm{rel}_{\press}$ for $\poroOpLipschitz = \frac45$ and $\pressOpLipschitz = \frac45$ corresponding to the solution of~\eqref{eq:geoTest} shown in \cref{fig:geoTest2b} with and without Lipschitz constant update.}}\label{fig:geoTest2b_rel_err}
\end{figure}%
We can do the same for the more applied model~\eqref{eq:geoTest}, \REV{with approximate solution shown in \cref{fig:geoTest2b}}.
\REV{For simplicity, we consider the error on a single time slice without propagated errors from previous slices, for which we have the estimates \eqref{eq:propagationphi} and \eqref{eq:propagationu}.}
The resulting error plots for a discontinuous $\poro_0$ can be found in \cref{fig:geoTest2b_rel_err}.
%
As in the previous case we observe \REV{optimal} rates.
%

\section{Viscous Limit}\label{sec:viscous}
A common simplification of~\eqref{eq:modelgeneral} is the viscous limit corresponding to $Q\to0$, leading to the equations
\begin{subequations}\label{eq:modelgeneralviscous}
	\begin{align}
		\partial_t \poroGeneral &= -{\beta(\poroGeneral)}{\kappa(\press)} ,\label{eq:modelgeneralviscous_a}\\
		0 &= \nabla \cdot \alpha(\poroGeneral) (\nabla \press + \zeta(\poroGeneral))-{\beta(\poroGeneral)}{\kappa(\press)}.\label{eq:modelgeneralviscous_b}
	\end{align}
\end{subequations}
As before we write~\eqref{eq:modelgeneralviscous_a} in integral form and consider~\eqref{eq:modelgeneralviscous_b} in weak formulation,
\begin{subequations}\label{eq:viscousmildweak}
	\begin{align}
		\poroGeneral(t,\cdot) &= \poroGeneral_0 - \int_0^t \beta(\poroGeneral)\kappa(\press) \D s, & & \text{for $t \in [0,T]$}, \label{eq:viscousmild} \\
		0&=\nabla \cdot \alpha(\poroGeneral) (\nabla \press + \zeta(\poroGeneral))-\beta(\poroGeneral)\kappa(\press) & & \text{in \REV{$H^{-1}(\Omega)$}}. \label{eq:viscousweak}
	\end{align}
\end{subequations}
Furthermore, we assume that \cref{ass:sigma,ass:alphabeta} are satisfied, leading to similar linearizations as the ones introduced in \cref{sec:fixedpoint}. Well-posedness of this approach is shown in \cite[Sec.~3]{Bachmayr2023}.

A space-time adaptive numerical method similar to the one considered above can be obtained along similar lines in this case. Although~\eqref{eq:viscousweak} is elliptic, it is nonetheless time-dependent due to the coupling with $\poroGeneral$. As before, we linearize~\eqref{eq:viscousweak} by means of
\begin{align}\label{eq:viscousfixedpoint_press}
	0
	=\nabla \cdot \alpha(\poroGeneral) (\nabla \press^{(k)} + \zeta(\poroGeneral))-\beta(\poroGeneral) \frac{\press^{(k)}}{\sigma(\press^{(k-1)})},
\end{align}
given the previous iterate $\press^{(k-1)}$. As described in \cref{rem:taylor} a different linearization is possible here as well. Then we define
\begin{align*}
	\pressSolSp
	= \left\{ (\press,\pressFlux) \in L_2(0,T;H_0^1(\Omega)) \times L_2(\Omega_T)^d \,\colon\, \Div_x \pressFlux \in L_2(\Omega_T) \right\},
\end{align*}
with the induced graph norm
\begin{align*}
	\norm{(\press,\pressFlux)}_\pressSolSp^2
	= \norm{(\press,\pressFlux)}_{L_2(\Omega_T,\R^{d+1})}^2 + \norm{\nablax \press}_{L_2(\Omega_T,\R^d)}^2 + \norm{\Div_x \pressFlux}_{L_2(\Omega_T)}^2.
\end{align*}
Next we set
$
	\pressRanSp
	= L_2(\Omega_T) \times L_2(\Omega_T,\R^d)
$
with its canonical norm and for each fixed $\overline{\press}$ define
\begin{align}\label{eq:viscousoperator}
	\pressLinOp[\overline{\press}](\press,\pressFlux) = \begin{pmatrix} \Div_x \pressFlux + \beta \, \frac{\press}{\sigma(\overline{\press})}\\\pressFlux + \alpha \, \nablax \press\end{pmatrix},\quad
	\pressLinRhs = \begin{pmatrix}0\\- \alpha \, \zeta\end{pmatrix},
\end{align}
This allows us to rewrite~\eqref{eq:viscousfixedpoint_press} as
\begin{align}\label{eq:viscoussystemlin}
	\pressLinOp[\press^{(k-1)}](\press^{(k)},\pressFlux^{(k)}) = \pressLinRhs.
\end{align}
To solve~\eqref{eq:viscoussystemlin} numerically for given $\overline{\press}$, we compute
\begin{align*}
	(\pressDiscr,\pressFluxDiscr)
	= \argmin_{(\pressAltDiscr,\pressFluxAltDiscr) \in \pressSolSpDiscr} \norm{\pressLinOp[\overline{\press}](\pressAltDiscr,\pressFluxAltDiscr) - \pressLinRhs}_\pressRanSp,
\end{align*}
as before. Well-posedness and convergence of the adaptive solver can be shown more easily than above for the general case. In \cref{fig:geoTest2b_viscous}, we show a numerical test similar to the one from~\eqref{eq:geoTest} shown in \cref{fig:geoTest2b}, but now with $Q=0$.
\begin{figure}[h!]
	\centering
	\vspace*{-0.3cm}
	\capstart
	\includegraphics[width=\convfac\textwidth]{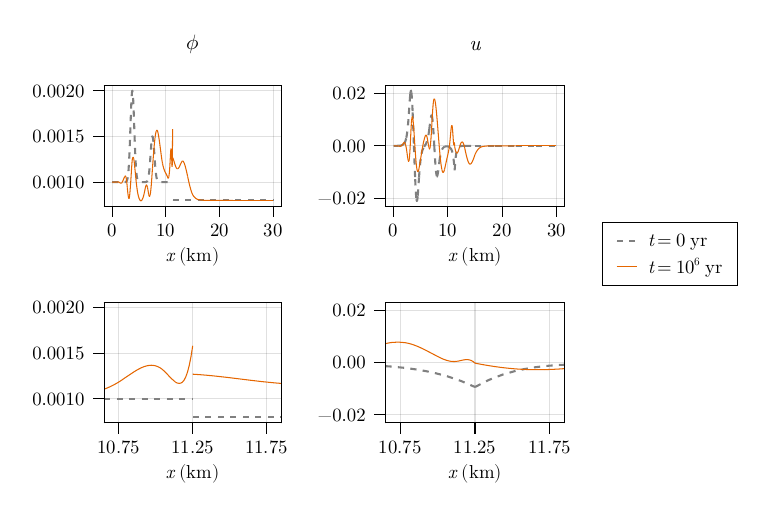}
	\vspace*{-0.6cm}
	\caption{Numerical approximation of $\poro$ and $\press$ for discontinuous $\poro_0$.}\label{fig:geoTest2b_viscous}
\end{figure}%
This time we show the numerical solution at the start and after 15 time slices, and as before we show a detail view near the discontinuity at the bottom of \cref{fig:geoTest2b_viscous}.
The corresponding grids are plotted in \cref{fig:geoTest2b_viscous_grid}.
\begin{figure}[h!]
	\centering
	\vspace*{-0.3cm}
	\capstart
	\FPmul\result{0.9}{\convfac}
	\includegraphics[width=\result\textwidth]{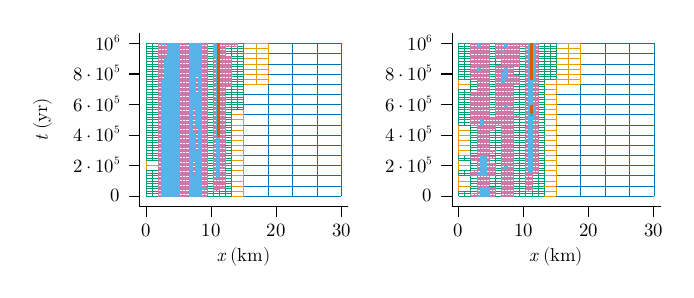}
	\vspace*{-0.6cm}
	\caption{Space-time grids for $\poro$ (left) and $\press$ (right) corresponding to \cref{fig:geoTest2b_viscous}.}\label{fig:geoTest2b_viscous_grid}
\end{figure}%

\subsection*{Acknowledgements}
The authors would like to thank Evangelos Moulas for introducing them to the models discussed in this work and for helpful discussions, Igor Voulis for advice on aspects of the implementation and Henrik Eisenmann for help concerning the results about perturbed fixed-point iterations.

\appendix

\REV{
\section{Model setup}\label{sec:setup}
Now we provide details about the model setup for all figures in this work.
All considered equations are of the general form
\begin{subequations}\label{eq:comparison}
	\begin{align}
		\partial_t \poro &=-(1-\poro) \bl \poro^m \, \press + Q \partial_t \press \br ,\\
		\partial_t \press &= \frac{1}{Q} \bl \nabla \cdot  (c_\phi \poro)^3 (\nabla \press + (1-\poro) \textbf{e}_d)- \frac{\poro^m}{\sigma(u)} \, \press \br ,
	\end{align}
\end{subequations}
where $\textbf{e}_d$ denotes a unit vector and $c_\phi > 0$ is a constant. Moreover, we set
\begin{align*}
	\sigma(v)=1 - c_1 \bl 1 + \tanh \bl -\frac{v}{c_2} \br \br ,\quad v\in \R,
\end{align*}
similar to~\eqref{eq:sigma}.}
\REV{The choice of these constants as well as details on $\Omega$ and $T$ (both in nondimensional form) for all the numerical tests can be found in \cref{tab:model_params}.
\begin{table}[h!]
	\centering
	\REV{
	\begin{tabular}{lccccccc}
		\hline
		Figure & $\Omega$ & $T$ & $m$ & $Q$ & $c_\phi$ & $c_1$ & $c_2$ \\
		\hline
		\cref{fig:comparison} & $(0,3)$ & $157.788$ & $1$ & $\frac{1}{60}$ & $10$ & $0$ & $1$ \\
		\cref{fig:unphysical_physical} & $(0,1)$ & $3$ & $2$ & $1$ & $1$ & $\frac{199}{400}$ & $\frac{1}{200}$ \\
		\cref{fig:fullNonLin} & $(0,1)$ & $1$ & $1$ & $1$ & $1$ & $\frac{12}{25}$ & $\frac{1}{25}$ \\
		\cref{fig:geoTest2b} & $(0,3)$ & $157.788$ & $1$ & $\frac{1}{60}$ & $10$ & $0$ & $1$ \\
		\cref{fig:channel} & $(0,1)\times(0,2)$ & $236.682$ & $2$ & $\frac{1}{60}$ & $10$ & $\frac{999}{2000}$ & $\frac{1}{500}$ \\
		\cref{fig:geoTest2b_viscous} & $(0,3)$ & $157.788$ & $1$ & $0$ & $10$ & $0$ & $1$ \\
		\hline
	\end{tabular}}
	\caption{Model parameters}\label{tab:model_params}
\end{table}}

\REV{
Furthermore we provide all initial functions $\poro_0$ and $\press_0$ in \cref{tab:init_setup}.
\begin{table}[h!]
	\centering
	\REV{
	\begin{tabular}{llc}
		\hline
		\cref{fig:comparison} & $\poro_0(x) = \begin{cases}0.002 & \text{if } x \leq 1.5\\0.001 &\text{else}\end{cases}$ & $u_0(x) = 0$ \\
		\cref{fig:unphysical_physical} & $\poro_0(x) = \begin{cases}0.2 & \text{if } x \in [0.25,0.5]\\0.1 &\text{else}\end{cases}$ & $u_0(x) = 0$  \\
		\cref{fig:fullNonLin} & $\poro_0(x) = \begin{cases}0.2 & \text{if } x \in [0.25,0.75]\\0.1 &\text{else}\end{cases}$ & $u_0(x) = 0$ \\
		\cref{fig:geoTest2b} & $\poro_0(x) = \begin{cases}\frac{1}{1000} \, (\exp(-100(x-\frac{3}{8})^2)+\frac12\exp(-100(x-\frac{3}{4})^2)+1) & \text{if } x \leq \frac{9}{8}\\\frac{1}{1250} & \text{else}\end{cases}$ & $u_0(x) = 0$ \\
		\cref{fig:channel} & $\poro_0(x) = \begin{cases}\frac{1}{1000} \, (\exp(-100((x_1-\frac12)^2+(x_2-\frac12)^2))+1) & \text{if } x_2 \leq 1\\\frac{3}{2500} & \text{else}\end{cases}$ & $u_0(x) = 0$ \\
		\cref{fig:geoTest2b_viscous} & $\poro_0(x) = \begin{cases}\frac{1}{1000} \, (\exp(-100(x-\frac{3}{8})^2)+\frac12\exp(-100(x-\frac{3}{4})^2)+1) & \text{if } x \leq \frac{9}{8}\\\frac{1}{1250} & \text{else}\end{cases}$ & -- \\
		\hline
	\end{tabular}}
	\caption{Initial functions}\label{tab:init_setup}
\end{table}}

\REV{
Finally, we also provide the tolerances used to run the methods from \cref{alg:coupled,alg:adapt_tol} in \cref{tab:alg_params}, where the initial grid size denotes the number of elements in spatial and temporal direction of each initial grid. For example $4\times4$ refers to $4$ elements in spatial direction and $4$ element in temporal direction which means that this grid has $16$ elements in total.
\begin{table}[h!]
	\centering
	\REV{
	\begin{tabular}{lccccccccc}
		\hline
		Figure & initial grid size & \#(time slices) & $\tol_\poroGeneral$ & $\tol_\mathrm{proj}$ & $\tol_\press$ & $\tol_\mathrm{lsq}$ & $\errorRedFac_\poroGeneral$ & $\errorRedFac_\press$ & $\theta$ \\
		\hline
		\cref{fig:comparison} & $8\times1$ & 20 & $10^{-5}$ & $10^{-7}$ & $10^{-5}$ & $5\cdot10^{-7}$ & -- & -- & -- \\
		\cref{fig:unphysical_physical} & $4\times1$ & 50 & $5\cdot10^{-5}$ & $10^{-6}$ & $10^{-4}$ & $5\cdot10^{-6}$ & -- & -- & -- \\
		\cref{fig:fullNonLin} & $4\times4$ & 1 & $10^{-4}$ & $5\cdot10^{-6}$ & $10^{-4}$ & $5\cdot10^{-6}$ & -- & -- & -- \\
		\cref{fig:geoTest2b} & $8\times1$ & 10 & $10^{-6}$ & $5\cdot10^{-7}$ & $10^{-6}$ & $10^{-7}$ & -- & -- & -- \\
		\cref{fig:channel} & $4\times8\times1$ & 15 & $10^{-5}$ & $2\cdot10^{-7}$ & $2\cdot10^{-6}$ & $10^{-6}$ & -- & -- & -- \\
		\cref{fig:fullNonLinErr} & $4\times4$ & 1 & $10^{-5}$ & -- & -- & -- & $\frac12$ & $\frac15$ & $\frac12$ \\
		\cref{fig:geoTest2b_rel_err} & $8\times1$ & 1 & $5\cdot10^{-7}$ & -- & -- & -- & $\frac{1}{100}$ & -- & $\frac12$ \\
		\cref{fig:geoTest2b_viscous} & $8\times1$ & 15 & $5\cdot10^{-7}$ & $10^{-7}$ & $5\cdot10^{-7}$ & $5\cdot10^{-8}$ & -- & -- & -- \\
		\hline
	\end{tabular}}
	\caption{Algorithm parameters}\label{tab:alg_params}
\end{table}}

\REV{
\section{Parabolic H\"{o}lder norms}\label{sec:holder}
On the partition $\Omega^j$, $j=1,\ldots,M$, defined in \cref{rem:piecewise}, we introduce parabolic H\"{o}lder spaces $C^{k,\gamma}_\mathrm{par}(\overline{\Omega}_T^j)$ for $k \in \{0,1\}$ and $\gamma \in (0,1)$ endowed with the parabolic space-time norm defined in~\cite{Ladyzenskaja1968},
\begin{equation*}
	\begin{aligned}
		\norm{u}_{C^{0,\gamma}_\mathrm{par}(\overline{\Omega}^j_T)}
		= &\sup_{\overline{\Omega}^j_T} |u|+ \sup_{\substack{t,x_1,x_2\\x_1 \neq x_2}} \frac{\AV{ u(t,x_1) - u(t,x_2)}}{\AV{x_1 - x_2}^\gamma} + \sup_{\substack{t_1,t_2,x\\t_1 \neq t_2}} \frac{\AV{ u(t_1,x) - u(t_2,x)}}{\AV{t_1 - t_2}^{\frac{\gamma}{2}}} ,\\
		\norm{u}_{C^{1,\gamma}_\mathrm{par}(\overline{\Omega}^j_T)}
		= &\sup_{\overline{\Omega}^j_T} |u| +
		\sup_{\overline{\Omega}^j_T} |\nabla u| + \sup_{\substack{t,x_1,x_2\\x_1 \neq x_2}} \frac{\AV{\nabla u(t,x_1) - \nabla u(t,x_2)}}{\AV{x_1 - x_2}^\gamma}\\
		&+ \sup_{\substack{t_1,t_2,x\\t_1 \neq t_2}} \frac{\AV{u(t_1,x) - u(t_2,x)}}{\AV{t_1 - t_2}^{\frac{1 + \gamma}{2}}} 
		+ \sup_{\substack{t_1,t_2,x\\t_1 \neq t_2}} \frac{\AV{\nabla u(t_1,x) - \nabla u(t_2,x)}}{\AV{t_1 - t_2}^{\frac{\gamma}{2}}} .
	\end{aligned}
\end{equation*}
}

\bibliographystyle{plain}

\end{document}